\documentclass[runningheads]{llncs}
\usepackage{graphicx}
\usepackage{amsmath}
\usepackage{amsfonts}
\usepackage{enumerate}
\usepackage{amssymb,cite}
\usepackage[title]{appendix}
\usepackage{environ}
\newcommand{\1}{P_1}
\newcommand{\2}{P_2}
\newcommand{\F}{\mathcal{F}}

\newcommand{\Sep}{\mathcal{S}}
\usepackage{tikz}
\usepackage{graphicx}
\usepackage{float}
\usetikzlibrary{arrows}
\usetikzlibrary{shapes}

\usepackage{todonotes}

\spnewtheorem{observation}[theorem]{Observation}{\bfseries}{\itshape}

\spnewtheorem*{ramseysTheorem}{Ramsey's Theorem}{\bf}{\it}
\newcommand{\repeattheorem}[1]{%
  \begingroup
  \renewcommand{\thetheorem}{\ref{#1}}%
  \expandafter\expandafter\expandafter\theorem
  \csname reptheorem@#1\endcsname
  \endtheorem
  \endgroup
}

\NewEnviron{reptheorem}[1]{%
  \global\expandafter\xdef\csname reptheorem@#1\endcsname{%
    \unexpanded\expandafter{\BODY}%
  }%
  \expandafter\theorem\BODY\unskip\label{#1}\endtheorem
}

\newcommand{\repeatlemma}[1]{%
  \begingroup
  \renewcommand{\thelemma}{\ref{#1}}%
  \expandafter\expandafter\expandafter\lemma
  \csname replemma@#1\endcsname
  \endtheorem
  \endgroup
}

\NewEnviron{replemma}[1]{%
  \global\expandafter\xdef\csname replemma@#1\endcsname{%
    \unexpanded\expandafter{\BODY}%
  }%
  \expandafter\lemma\BODY\unskip\label{#1}\endlemma
}
\newcommand{\repeatcor}[1]{%
  \begingroup
  \renewcommand{\thecorollary}{\ref{#1}}%
  \expandafter\expandafter\expandafter\corollary
  \csname repcorollary@#1\endcsname
  \endcorollary
  \endgroup
}

\NewEnviron{repcorollary}[1]{%
  \global\expandafter\xdef\csname repcorollary@#1\endcsname{%
    \unexpanded\expandafter{\BODY}%
  }%
  \expandafter\corollary\BODY\unskip\label{#1}\endcorollary
}

\newcommand{\repeatprop}[1]{%
  \begingroup
  \renewcommand{\theproposition}{\ref{#1}}%
  \expandafter\expandafter\expandafter\proposition
  \csname repproposition@#1\endcsname
  \endproposition
  \endgroup
}

\NewEnviron{repproposition}[1]{%
  \global\expandafter\xdef\csname repproposition@#1\endcsname{%
    \unexpanded\expandafter{\BODY}%
  }%
  \expandafter\proposition\BODY\unskip\label{#1}\endproposition
}

\newcommand{\repeatobs}[1]{%
  \begingroup
  \renewcommand{\theobservation}{\ref{#1}}%
  \expandafter\expandafter\expandafter\observation
  \csname repobservation@#1\endcsname
  \endobservation
  \endgroup
}

\NewEnviron{repobservation}[1]{%
  \global\expandafter\xdef\csname repobservation@#1\endcsname{%
    \unexpanded\expandafter{\BODY}%
  }%
  \expandafter\observation\BODY\unskip\label{#1}\endobservation
}

\begin{document}
\title{Minimal separators in graph classes\\ defined by small forbidden induced subgraphs\thanks{This work is supported in part by the Slovenian Research Agency (I0-0035, research program P1-0285, research projects N1-0032, J1-7051, J1-9110, and a Young Researchers grant). Part of the work was done while the first named author was visiting Osaka Prefecture University in Japan, under the operation Mobility of Slovene higher education teachers 2018--2021, co-financed by the Republic of Slovenia and the European Union under the European Social Fund.}}
\titlerunning{Minimal separators in hereditary graph classes}
%
\author{Martin Milani\v c\inst{1,2}\and
Nevena Piva\v c\inst{1}}
\authorrunning{M. Milani\v c, N. Piva\v c}
%
\institute{University of Primorska, IAM, Muzejski trg 2, SI-6000 Koper, Slovenia\\
\email{martin.milanic@upr.si}, \email{nevena.pivac@iam.upr.si}\\
\and
University of Primorska, FAMNIT, Glagolja\v ska 8, SI-6000 Koper, Slovenia\\
}

%
\maketitle 
\begin{sloppypar}
Minimal separators in graphs are an important concept in algorithmic graph theory. In particular, many problems that are NP-hard for general graphs are known to become polynomial-time solvable for classes of graphs with a polynomially bounded number of minimal separators. Several well-known graph classes have this property, including chordal graphs, permutation graphs, circular-arc graphs, and circle graphs. We perform a systematic study of the question which classes of graphs defined by small forbidden induced subgraphs have a polynomially bounded number of minimal separators. We focus on sets of forbidden induced subgraphs with at most four vertices and obtain an almost complete dichotomy, leaving open only two cases.
\end{sloppypar}
\section{Introduction} \label{sec:introduction}

\begin{sloppypar}
The main concept studied in this paper is that of a minimal separator in a graph. Given a graph $G$, a \emph{minimal separator} in $G$ is a subset of vertices that separates some non-adjacent vertex pair $a,b$ and is inclusion-minimal with respect to this property (separation of $a$ and $b$).
Minimal separators in graphs are important for reliability analysis of networks~\cite{MR592113}, for sparse matrix computations, via their connection with minimal triangulations (see~\cite{MR2204109} for a survey), and are related to other combinatorial graph concepts such as potential maximal cliques~\cite{MR1857397}. Many graph algorithms are based on minimal separators, see, e.g.,~\cite{MR1361391,MR1478250,MR1451636,MR1658776,MR1632642,MR1723686,MR1857397,MR1978261,MR2473933,MR2583286,MR2839718,CM2019}.
\end{sloppypar}

\begin{sloppypar}
In this work we focus on graphs with ``few'' minimal separators. Such graphs enjoy good algorithmic properties. Many problems that are NP-hard for general graphs become polynomial-time solvable for classes of graphs with a polynomially bounded number of minimal separators. This includes 
{\sc Treewidth} and {\sc Minimum Fill-In}~\cite{MR1896345}, {\sc Maximum Independent Set} and {\sc Feedback Vertex Set}~\cite{MR2853937}, {\sc Distance-$d$ Independent Set} for even $d$~\cite{MR3593955} and many other problems~\cite{MR3311877}. It is therefore important to identify classes of graphs with a polynomially bounded number of minimal separators.
Many known graphs classes have this property, including chordal graphs~\cite{MR0408312}, chordal bipartite graphs~\cite{Kratsch}, weakly chordal graphs~\cite{MR1857397}, permutation graphs~\cite{MR1312164,MR1361391}, circular-arc graphs~\cite{Kratsch,MR1723686,MR1632642}, circle graphs~\cite{kloks1996treewidth,Kratsch,MR1632642}, etc. Moreover, a class of graphs has a polynomially bounded number of minimal separators if and only if it has a polynomially bounded number of potential maximal cliques~\cite{MR1896345}.
\end{sloppypar}

We perform a systematic study of the question which classes of graphs defined by small forbidden induced subgraphs have a polynomially bounded number of minimal separators. We focus on sets of forbidden induced subgraphs with at most four vertices and obtain an almost complete dichotomy, leaving open only two cases, the class of graphs of independence number three that are either \mbox{$C_4$-free} of $\{\textrm{claw}, C_4\}$-free. Our approach combines a variety of tools and techniques, including constructions of graph families with exponentially many minimal separators, applications of Ramsey's theorem, study of the behavior of minimal separators under various graph operations, and structural characterizations of graphs in hereditary classes.

\medskip
\noindent\textbf{Statement of the main result.} Given two non-adjacent vertices $a$ and $b$ in a graph $G$, a set $S\subseteq V(G)\setminus \{a,b\}$ is an \emph{$a,b$-separator} if $a$ and $b$ are contained in different components of $G-S$. If $S$ contains no other $a,b$-separator as a proper subset, then $S$ is a \emph{minimal $a,b$-separator}. We denote by $\Sep_G(a,b)$ the set of all minimal $a,b$-separators. A \emph{minimal separator} in $G$ is a set $S\subseteq V(G)$ that is a minimal $a,b$-separator for some pair of non-adjacent vertices $a$ and $b$. We denote by $\Sep_G$ the set of all minimal separators in $G$ and by $s(G)$ the cardinality of $\Sep_G$. The main concept of study in this paper is the following property of graph classes.

\begin{definition}\label{def:poly}
We say that a graph class $\mathcal{G}$ is \emph{tame} if
there exists a polynomial $p:\mathbb{R}\to\mathbb{R}$ such that for every graph $G\in \mathcal{G}$, we have $s(G)\le p(|V(G)|)$.
\end{definition}

Given a family $\F$ of graphs, we say that a graph $G$ is \emph{$\F$-free} if no induced subgraph of $G$ is isomorphic to a member of $\F$. Given two families $\F$ and $\F'$ of graphs, we write $\F\trianglelefteq \F'$
if the class of $\F$-free graphs is contained in the class of $\F'$-free graphs, or, equivalently, if every $\F$-free graph is also $\F'$-free.

\begin{observation}\label{obs:trivial}
Let $\F$ and $\F'$ be two graph families such that $\F\trianglelefteq \F'$.
If the class of $\F'$-free graphs is tame, then the class of $\F$-free graphs is tame.
\end{observation}

It is well known and not difficult to see that relation $\F\trianglelefteq \F'$ can be checked by means of the following criterion, which becomes particularly simple for finite families $\F$ and $\F'$.

\begin{observation}[Folklore]
For every two graph families $\F$ and $\F'$, we have $\F\trianglelefteq \F'$
if and only if every graph from $\F'$
contains an induced subgraph isomorphic to a member of $\F$.
\end{observation}

Our main result is Theorem~\ref{thm:main}. It deals with graph classes defined by sets of forbidden induced subgraphs having at most four vertices. The relevant graphs are named as in Fig.~\ref{fig:allgraphs}.

\begin{figure}[H]
\centering
\includegraphics[width=\linewidth]{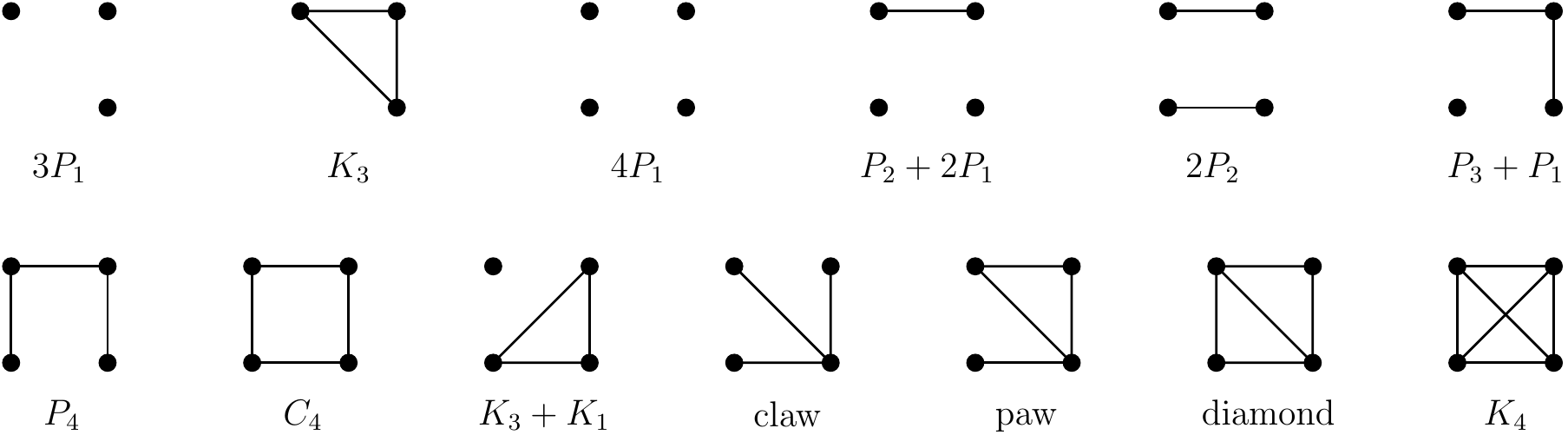}
\caption{Graphs on at most $4$ vertices appearing in the statement of the main theorem.}
\label{fig:allgraphs}
\end{figure}

\begin{reptheorem}{theorem:main}
\label{thm:main}
Let $\F$ be a family of graphs with at most four  vertices such that $\F\neq \{4\1, C_4\}$ and $\F\neq \{4\1$, claw, $C_4\}$. Then the class of $\F$-free graphs is tame if and only if $\F\trianglelefteq \F'$ for one of the following families $\F'$:
\begin{enumerate}[i)]
\item \label{item1} $\F'=\{P_4\}$ or $\F'=\{2\2\}$,
\item $\F'=\{F, \textrm{paw}\}$ for some $F\in \{4\1$, $\2+2\1$, $P_3+\1$, claw$\}$,
\item $\F'=\{F, K_3+\1\}$ for some $F\in \{4\1$, $\2+2\1$, $P_3+\1$, claw$\}$,
\item $\F'=\{F, K_4\}$ for some $F\in \{4\1$, $\2+2\1$, $P_3+\1\}$,
\item $\F'=\{F, C_4\}$ for some $F\in \{ \2+ 2\1$, $P_3+ \1\}$,
\item \label{item-last} $\F'=\{4\1, C_4$, diamond$\}$.
\end{enumerate}
\end{reptheorem}

Theorem~\ref{thm:main} can be equivalently stated in a dual form, characterizing minimal classes of $\F$-graphs that are not tame.

\begin{reptheorem}{theorem:main-dual}
\label{thm:main-dual}
Let $\F$ be a family of graphs with at most $4$ vertices such that
$\F\neq \{4\1,C_4\}$ and $\F\neq \{4\1$, $C_4$, claw$\}$. Then the class of $\F$-free graphs is not tame if and only if $\F'\trianglelefteq \F$ for one of the following families $\F'$:
\begin{enumerate}[i)]
\item \label{item1-dual}  $\F'=\{3\1$, diamond$\}$,
\item \label{item2-dual}  $\F'=\{$claw, $K_4$, $C_4$, diamond$\}$,
\item \label{item-last-dual} $\F'=\{K_3$, $C_4\}$.
\end{enumerate}
\end{reptheorem}

In Figure~\ref{fig:mainresult} we give an overview of maximal tame and minimal non-tame classes of $\mathcal{F}$-free graphs, where $\mathcal{F}$ contains graphs with at most four vertices. Maximal tame classes correspond to sets $\mathcal{F}$ of forbidden induced subgraphs depicted in green ellipses, while minimal non-tame classes correspond to sets depicted in red ellipses (in gray-scale printing, red and green ellipses appear as brighter, resp., darker ellipses). The two remaining open cases are dashed. A similar figure with respect to boundedness of the clique-width can be found in~\cite{MR2237514}.

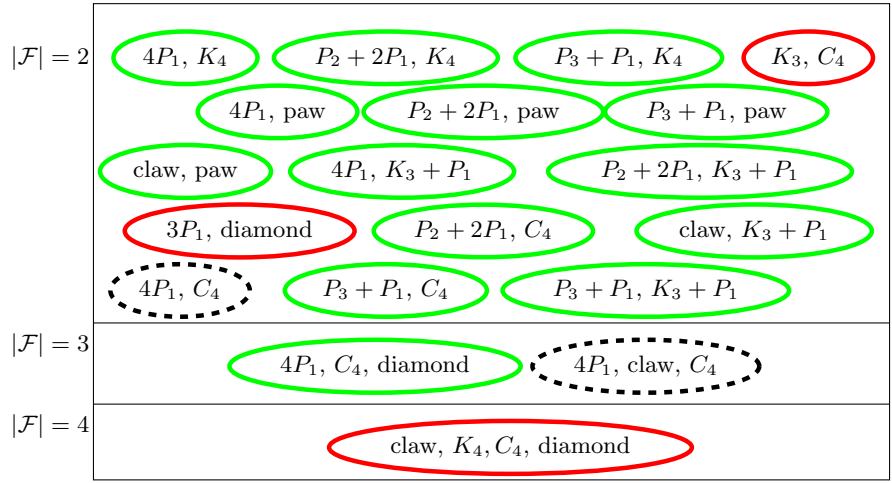
\begin{figure}
\centering
\begin{tikzpicture}[scale=0.72]
\draw (-0.7,1.1) -- (14,1.1) -- cycle;
\draw (-0.7, -0.4) -- (14,-0.4) -- cycle;
\draw (-0.7,-1.8) -- (14,-1.8) -- (14, 7) -- (-0.7,7) -- cycle;

\node at (-1.5,6) {$|{\cal F}|=2$};
\node at (-1.5,0.7) {$|{\cal F}|=3$};
\node at (-1.5,-0.8) {$|{\cal F}|=4$};

\node [ellipse, line width=0.6mm, draw=green, , minimum size=0.5cm] at (1,6) {$4\1$, $K_4$};
\node [ellipse, line width=0.6mm, draw=green, , minimum size=0.5cm] at (4.7,6) {$\2+2\1$, $K_4$};
\node [ellipse, line width=0.6mm, draw=green, , minimum size=0.5cm] at (9,6) {$P_3+\1$, $K_4$};
\node [ellipse, line width=0.6mm, draw=red, , minimum size=0.5cm] at (12.5,6) {$K_3$, $C_4$};

\node [ellipse, line width=0.6mm, draw=green, , minimum size=0.5cm] at (2.7,5) {$4\1$, paw};
\node [ellipse, line width=0.6mm, draw=green, , minimum size=0.5cm] at (6.5,5) {$\2+2\1$, paw};
\node [ellipse, line width=0.6mm, draw=green, , minimum size=0.5cm] at (10.8,5) {$P_3+\1$, paw};

\node [ellipse, line width=0.6mm, draw=green, , minimum size=0.5cm] at (1,3.9) {claw, paw};

\node [ellipse, line width=0.6mm, draw=green, , minimum size=0.5cm] at (5,3.9) {$4\1$, $K_3+\1$};
\node [ellipse, line width=0.6mm, draw=green, , minimum size=0.5cm] at (10.5,3.9) {$\2+2\1$, $K_3+\1$};
\node [ellipse, line width=0.6mm, draw=green, , minimum size=0.5cm] at (11.5,2.8) {claw, $K_3+\1$};

\node [ellipse, line width=0.6mm, draw=green, , minimum size=0.5cm] at (6.5,2.8) {$ \2+2\1$, $C_4$};
\node [ellipse, line width=0.6mm, draw=red, , minimum size=0.5cm] at (2,2.8) {$3\1$, diamond};
\node [ellipse, line width=0.6mm, draw=green, , minimum size=0.5cm] at (9.5,1.7) {$P_3+\1$, $K_3+\1$};
\node [ellipse, line width=0.6mm, draw=green, , minimum size=0.5cm] at (4.7,1.7) {$ P_3+\1$, $C_4$};
\node [ellipse, dashed, line width=0.6mm, draw=black, , minimum size=0.5cm] at (0.9,1.7) {$4\1$, $C_4$};

\node [ellipse, dashed, line width=0.6mm, draw=black, , minimum size=0.5cm] at (9.5,0.3) {$4\1$, claw, $C_4$};

\node [ellipse, line width=0.6mm, draw=green, , minimum size=0.5cm] at (4.5,0.3) {$4\1$, $C_4$, diamond};
\node [ellipse, line width=0.6mm, draw=red, , minimum size=0.5cm] at (7,-1.2) {claw, $K_4, C_4$, diamond};

\end{tikzpicture}
\caption{Overview of the main result.}
\label{fig:mainresult}
\end{figure}

\medskip
\noindent\textbf{Related work.}
To the best of the authors' knowledge, this work represents the first systematic study of the problem of classifying hereditary graph classes with respect to the existence of a polynomial bound on the number of minimal separators of the graphs in the class. Dichotomy studies for many other problems in mathematics and computer science are available in the literature
in general, as well as within the field of graph theory, for properties such as boundedness of the clique-width~\cite{MR2237514,MR3672793,MR3515312,MR3442572}, price of connectivity~\cite{MR3530629}, and polynomial-time solvability of various algorithmic problems such as {\sc Chromatic Number}~\cite{MR3575032,MR3623382,MR1905637}, {\sc Graph Homomorphism}~\cite{MR1047555}, {\sc Graph Isomorphism}~\cite{MR3712298}, and {\sc Dominating Set}~\cite{MR3515011}.

\medskip
\noindent\textbf{Structure of the paper.}
We collect the main notations, definitions, and preliminary results in Section~\ref{sec:prelim}. In Section~\ref{sec:exp-constructions} we
present several families of graphs with exponentially many minimal separators. In Section~\ref{sec:operations}, we study the effect of various graph operations on the number of minimal separators. Our main result, given by Theorems~\ref{thm:main} and~\ref{thm:main-dual}, is proved in Section~\ref{sec:proof-of-main-theorem}.

\section{Preliminaries} \label{sec:prelim}

All graphs in this paper will be finite, simple, undirected, and will have at least one vertex. A vertex $v$ in a graph $G$ is \emph{universal} if it is adjacent to every other vertex in the graph, and \emph{simplicial} if its neighborhood is a clique. Two vertices $u$ and $v$ are said to be \emph{true twins} if $N_G[u]=N_G[v]$. A graph $G$ is \emph{true-twin-free} if no two distinct vertices in $G$ are true twins. The \emph{maximum co-degree} of a graph $G$ is defined as the maximum number of non-neighbors of a vertex. A \emph{co-component} of a graph $G$ is the subgraph of $G$ induced by the vertex set of a component of~$\overline{G}$, where $\overline{G}$ denotes the complement of $G$. A graph is \emph{co-connected} if its complement is connected. Given two graphs $F$ and $G$, we write
$F\subseteq_i G$ if $F$ is an induced subgraph of $G$. A graph class $\mathcal{G}$ is \emph{hereditary} if it is closed under vertex deletion, or, equivalently, if there exists a set $\mathcal{F}$ of graphs such that $\mathcal{G}$ is exactly the class of $\mathcal{F}$-free graphs.

If $A$ and $B$ are disjoint subsets of $V(G)$, we say that they are \emph{complete} (resp., \emph{anticomplete}) \emph{to each other} in $G$ if $\{ab\mid a\in A, b\in B\} \subseteq E(G)$ (resp., $\{ab\mid a\in A, b\in B\}\cap E(G)=\emptyset$). A graph $G$ is the \emph{join} of two vertex-disjoint graphs $G_1$ and $G_2$, written $G = G_1\ast G_2$, if $V(G) = V(G_1)\cup V(G_2)$ and $E(G) = E(G_1)\cup E(G_2)\cup \{xy\mid x\in V(G_1) \, \textrm{and}\, y\in V(G_2)\}$. As usual, we denote by $P_n$, $C_n$, $K_n$ the path, the cycle, and the complete graph with $n$ vertices, respectively. For positive integers $m,n$, we denote by $K_{m,n}$ the complete bipartite graph with $m$ and $n$ vertices in the two parts of the bipartition. The \emph{claw} is the graph $K_{1,3}$, the \emph{co-claw} is the complement of the claw (that is, the graph $K_3+\1$). Given a non-negative integer $k$, the disjoint union of $k$ copies of $G$ is denoted by $kG$. The \emph{girth} of a graph $G$ is the smallest integer $k$ such that $C_k$ is an induced subgraph of $G$ (or $\infty$ if $G$ is acyclic). Given a graph $G$ and sets $A,B,W\subseteq V(G)$, we say that $W$ \emph{separates} $A$ and $B$ if $A$ and $B$ are contained in different components of $G-W$ (thus, in particular, the sets $A,B,W$ are pairwise disjoint). Given a graph $G$, its \emph{line graph} is the graph $L(G)$ with vertex set $E(G)$ in which two distinct vertices $e$ and $e'$ are adjacent if and only if $e$ and $e'$ have a common endpoint as edges in $G$. For undefined terms related to graphs and graph classes, we refer the reader to~\cite{MR1971502,MR2063679}.

An important ingredient for some of our proofs will be the following classical result~\cite{MR1576401}.

\begin{ramseysTheorem}
For every two positive integers $k$ and $\ell$, there exists a least positive
integer $R(k,\ell)$ such that every graph with at least
$R(k,\ell)$ vertices contains either a clique of size $k$ or an independent set of size $\ell$.
\end{ramseysTheorem}

Given a graph $G$ and a set $S\subseteq V(G)$, a component $C$ of the graph $G-S$ is \emph{$S$-full} if every vertex in $S$ has a neighbor in $C$, or, equivalently, if $N_G(S) = C$. The following well-known lemma characterizes minimal separators (see, e.g., \cite{MR1723686,MR2063679,MR1288579}).

\begin{lemma}
\label{lem:characterization-minimal-separators}
Given a graph $G = (V,E)$, a set $S\subseteq V$ is a minimal separator in $G$
if and only if the graph $G-S$ contains at least two $S$-full components.
\end{lemma}

%

\begin{corollary}\label{cor:deleting-a-vertex-in-S}
Let $S$ be a minimal separator in a graph $G$. Then for every $v\in S$ the set $S\setminus\{v\}$ is a minimal separator in $G-v$.
\end{corollary}

\begin{proof}
Let $G'  = G-v$ and $S' = S\setminus \{v\}$. Since $S$ is a minimal separator in $G$, there exists two $S$-full components $C$ and $D$ in $G-S$. Since $G-S = G'-S'$ and $S'\subseteq S$, it follows that $C$ and $D$ are also $S'$-full components of $G'-S'$. Hence, $S'$ is a minimal separator in $G'$.
\qed\end{proof}

The following result shows that the class of $P_4$-free graphs is tame.

\begin{theorem}[Nikolopoulos and Palios~\cite{MR2204116}]
\label{thm:cographs}
If $G$ is a $P_4$-free graph, then $s(G)< 2/3|V(G)|$.
\end{theorem}

\begin{sloppypar}
We will also need the following result about the structure of paw-free graphs. A graph $G$ is \emph{complete multipartite} if its vertex set can be partitioned into any number of parts such that two vertices are adjacent if and only if they belong to different parts.
\end{sloppypar}

\begin{theorem}[Olariu~\cite{MR947254}]
\label{thm:pawfree}
A connected paw-free graph $G$ is either $K_3$-free or complete multipartite.
\end{theorem}

We conclude this section with a straightforward but useful simplification of the defining property of tame graph classes.

\begin{replemma}{lemmaa:poly-monomial}
\label{lem:poly-monomial}
A graph class $\mathcal{G}$ is tame if and only if there exists a non-negative integer $k$ such that $s(G)\le |V(G)|^k$  for all $G\in \mathcal{G}$.
\end{replemma}
\begin{proof}
Sufficiency (the ``if'' direction) is trivial. To prove necessity (the ``only if'' direction), let $\mathcal{G}$ be tame graph class and let $p(x) = \sum_{i = 0}^d a_ix^i$ be a polynomial such that $s(G)\le p(|V(G)|)$ for all $G\in \mathcal{G}$. We may assume that $a_i\ge 0$ for all $i$, since otherwise we may delete the terms of $p$ with negative coefficients to obtain a polynomial $q$ such that $s(G)\le q(|V(G)|)$ for all $G\in \mathcal{G}$. Moreover, we may assume that $a_0 = \ldots = a_d$, since otherwise, as long as there exists a pair $(i,j)$ with $0\le i<j\le d$ and $a_i<a_j$, we may increase the $i$-th coefficient from $a_i$ to $a_j$ to obtain a polynomial $q$ such that $s(G)\le q(|V(G)|)$ for all $G\in \mathcal{G}$.
Let $a$ be this common value, that is, $a_0 = \ldots = a_d = a$.
We thus have $p(x) = a(\sum_{i = 0}^d x^i)$ and hence $p(n)\le a n^{d+1}$ holds for all $n\ge 2$. (Since the $1$-vertex graph has no minimal separators, we focus on $n\ge 2$.) Let $\ell$ be the least non-negative integer such that $a \le 2^\ell$. Then, for all $n\ge 2$, we have
$a n^{d+1}\le 2^\ell\cdot n^{d+1}\le n^\ell\cdot n^{d+1} = n^{d+\ell+1}$.
Taking $k = d+\ell+1$, necessity is proved. \qed
\end{proof}

An easy consequence of Lemma~\ref{lem:poly-monomial} is the fact that any union of finitely many tame graph classes is tame.

\section{Graph families with exponentially many minimal separators}
\label{sec:exp-constructions}

In this section we identify some families of graphs with exponentially many minimal separators. We give two constructions with structurally different properties. The first construction, explained in Section~\ref{sec:theta-graphs}, involves families of graphs of arbitrarily large maximum degree but without arbitrarily long induced paths. The second construction, explained in Section~\ref{sec:walls}, involves two families of graphs with small maximum degree but with arbitrarily long induced paths. In both cases, we make use of line graphs.

\subsection{Theta graphs and their line graphs}\label{sec:theta-graphs}

Given positive integers $k$ and $\ell$, the \emph{$k,\ell$-theta graph} is the graph $\theta_{k,\ell}$ obtained as the union of $k$ internally disjoint paths of length $\ell$ with common endpoints $a$ and $b$. For every positive integer $\ell$, we define a family of graphs $\Theta_{\ell}$ in the following way: $\Theta_{\ell} = \{\theta_{k,\ell}\mid k\ge 2\}$. Note that $\ell$ refers to the length of each of the $a,b$-paths and not to the number of paths, which is unrestricted.

\begin{repobservation}{observation:theta}
\label{obs:theta}
For every integer $\ell\ge 3$, the class $\Theta_{\ell}$ is not tame.
\end{repobservation}
\begin{proof}
Let $k\ge 2$, $\ell\ge 3$, let $G = \theta_{k,\ell}$, and let $P^1,\ldots, P^k$ be paths in $G$ as in the definition of the theta graphs. Let $S$ be any set of vertices of $G$ containing exactly one internal vertex of each of the paths $P^j$. Then, the graph $G-S$ has two $S$-full components and Lemma~\ref{lem:characterization-minimal-separators} implies that $S$ is a minimal separator in $G$. Note that for every $j\in \{1,\ldots, k\}$, path $P^j$ has exactly $\ell-1$ internal vertices. It follows that $s(\theta_{k,\ell})\ge (\ell-1)^k$. Thus, as $|V(\theta_{k,\ell})| = k(\ell-1)+2$, we infer that for every fixed positive integer $\ell\ge 3$, the class $\Theta_{\ell}$ is not tame. \qed
\end{proof}

\begin{corollary}\label{cor:theta}
If $\mathcal{G}$ is a class of graphs such that
$\Theta_{\ell}\subseteq \mathcal{G}$ for some $\ell\ge 3$, then $\mathcal{G}$ is not tame.
\end{corollary}

Consider now the family of line graphs of theta graphs. More precisely,
given positive integers $k$ and $\ell$, let $L_{k,\ell}$ denote the line graph of $\theta_{k,\ell}$ and let $\mathcal{L}_{\ell} = \{L_{k,\ell}\mid k\ge 2\}$.

\begin{repproposition}{observation:line-theta}
\label{prop:line-theta}
For every integer $\ell\ge 2$, the class $\mathcal{L}_{\ell}$ is not tame.
\end{repproposition}

\begin{proof}
Let $k,\ell\ge 2$ and let $G = L_{k,\ell}$. Then, graph $G$ consists of two cliques $K$ and $K'$, each of size $k$, say with $K = \{a_1,\ldots, a_k\}$ and
$K' = \{b_1,\ldots, b_k\}$, and $k$ internally pairwise disjoint paths $P^1,\ldots, P^k$ such that for every $j\in \{1,\ldots, k\}$,
path $P^j$ is an $a_i,b_i$-path with $|V(P^j)| = \ell$,
$V(P^j)\cap K = \{a_j\}$ and
$V(P^j)\cap K' = \{b_j\}$. Consider any set $S$ of vertices of $G$ containing exactly one vertex from each of the paths $P^j$ and such that
$S\notin \{K,K'\}$. Then, the graph
$G-S$ has two $S$-full components and Lemma~\ref{lem:characterization-minimal-separators} implies that $S$ is a minimal separator in $G$. It follows that $s(L_{k,\ell})\ge \ell^k-2$.
Thus, as $|V(L_{k,\ell})| = k(\ell+1)$, we infer that for every fixed positive integer $\ell\ge 2$, the class $\mathcal{L}_{\ell}$ is not tame. \qed
\end{proof}

\begin{corollary}\label{cor:line-theta}
If $\mathcal{G}$ is a class of graphs such that
${\cal L}_{\ell}\subseteq \mathcal{G}$ for some $\ell\ge 2$, then  $\mathcal{G}$ is not tame.
\end{corollary}

\begin{corollary}\label{cor:h2containedinf}
The class of $\{3\1$, diamond$\}$-free graphs is not tame.
\end{corollary}

\subsection{Elementary walls and their line graphs}\label{sec:walls}

Let $r,s\ge 2$ be integers. An \emph{$r\times s$-grid} is the graph with vertex set $\{0,\ldots, r-1\}\times \{0,\ldots, s-1\}$ in which two vertices $(i, j)$ and $(i',j')$ are adjacent if and only if $|i-i'|+|j-j'| = 1$.
Given an integer $h\ge 2$, an \emph{elementary wall of height $h$} is the graph $W_h$ obtained from the $(2h+2)\times (h+1)$-grid by deleting all edges with endpoints $(2i + 1, 2j)$ and $(2i + 1, 2j+1)$ for all $i\in \{0,1,\ldots, h\}$ and $j \in \{0,1,\ldots, \lfloor (h-1)/2\rfloor\}$, deleting all edges with endpoints $(2i, 2j-1)$ and $(2i, 2j)$ for all $i\in \{0,1,\ldots, h\}$ and $j \in \{1,\ldots, \lfloor h/2\rfloor\}$, and deleting the two resulting vertices of degree one. Note that an elementary wall of height $h$ consists of $h$ levels each containing $h$ bricks, where a brick is a cycle of length six; see Fig.~\ref{fig:wall}(a).

\begin{figure}[ht!]
\begin{center}
\includegraphics[width=\textwidth]{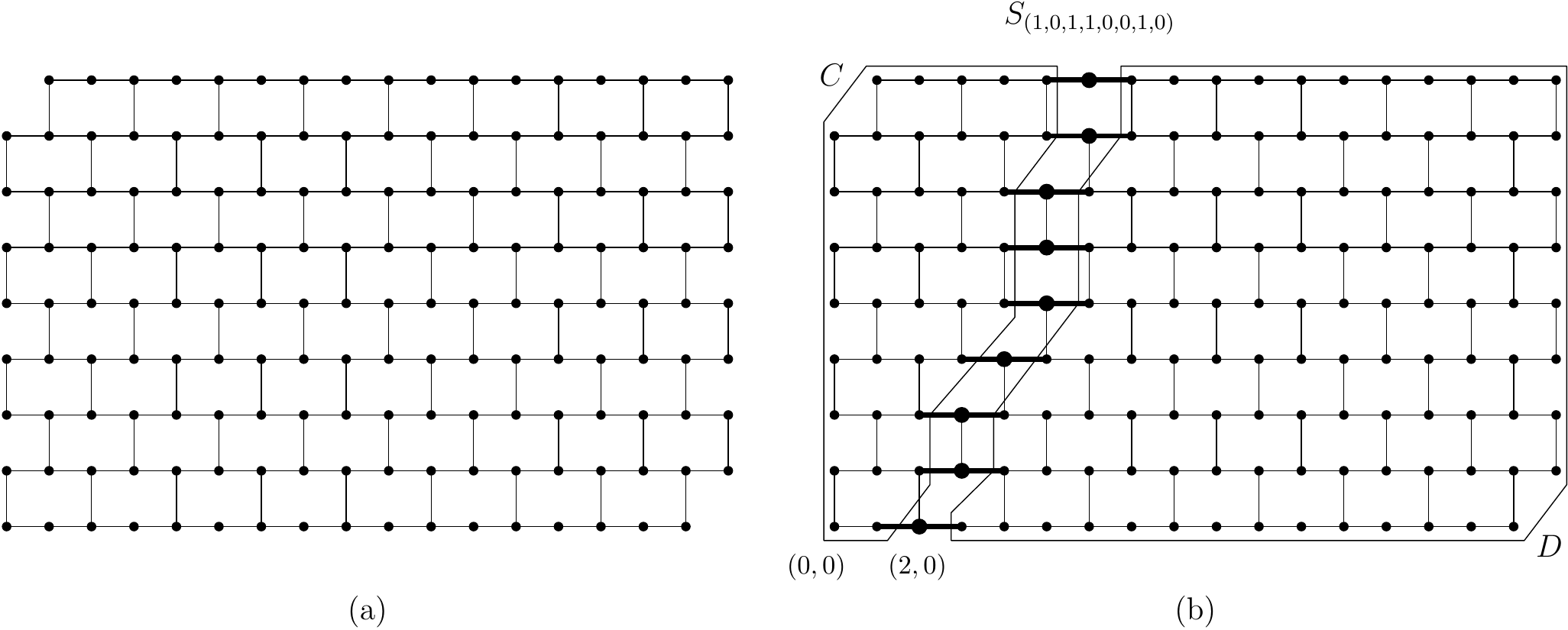}
\end{center}
\caption{(a) An elementary wall of height $8$. (b) A minimal separator $S_{(1,0,1,1,0,0,1,0)}$ in $W_8$ 
and the two components of $W_8-S_{(1,0,1,1,0,0,1,0)}$.}\label{fig:wall}
\end{figure}

Grids contain exponentially many minimal separators~\cite{Suchan}. A similar construction works for walls.

\begin{repproposition}{proposition:walls}
\label{prop:walls}
For every integer $h\ge 2$, an elementary wall of height $h$ has at least $2^h$ minimal separators.
\end{repproposition}

\begin{proof}
Fix an integer $h\ge 2$. We will define a family of $2^h$ subsets of $V(W_h)$ and show that each of them is a minimal separator in $W_h$. For each binary sequence of length $h$, say $x= (x_1,\ldots, x_h)\in \{0,1\}^h$, we define a set $S_x$ by the following rule: $S_x = \{v^{x,0}, v^{x,1}, \ldots, v^{x,h}\}$ where $v^{x,0} = (2,0)$ (independently of $x$) and for all $j \in \{1,\ldots, h\}$, we set $v^{x,j} = v^{x,j-1}+(x_j,1)$, where addition is performed component-wise.
Clearly, for each $x\in \{0,1\}^h$ and each $j\in \{1,\ldots, h\}$, we have  $v^{x,j} = (\sum_{i= 1}^jx_i+2,j)\le (h+2,h)$, where comparison is performed component-wise. It follows that $S_x\subseteq V(W_h)$. Moreover, the graph $W_h-S_x$ has exactly two connected components, say $C$ and $D$, with $V(C) = \bigcup_{j = 0}^h\{(i,j)\in V(W_h)\mid i<v^{x,j}_1\}$ and $V(D) = \bigcup_{j = 0}^h\{(i,j)\in V(W_h)\mid i>v^{x,j}_1\}$. Note that each vertex $v^{x,j}\in S_x$ has a neighbor in $C$, namely $v^{x,j}-(1,0)$, and a neighbor in $D$, namely
$v^{x,j}+(1,0)$. By Lemma~\ref{lem:characterization-minimal-separators}, set $S_x$ is a minimal separator in $W_h$. Since the sets $S_x$ are pairwise distinct, this completes the proof. Fig.~\ref{fig:wall}(b) shows an example with $h = 4$ and $x = (1,0,1,1)$. The thick horizontal edges can be used to justify the fact that $C$ and $D$ are $S_x$-full components of $W_h-S_x$. \qed
\end{proof}

Another useful family with exponentially many minimal separators is given by the line graphs of elementary walls; see Fig.~\ref{fig:line-wall}(a) for an example.

\begin{figure}[ht!]
\begin{center}
\includegraphics[width=\textwidth]{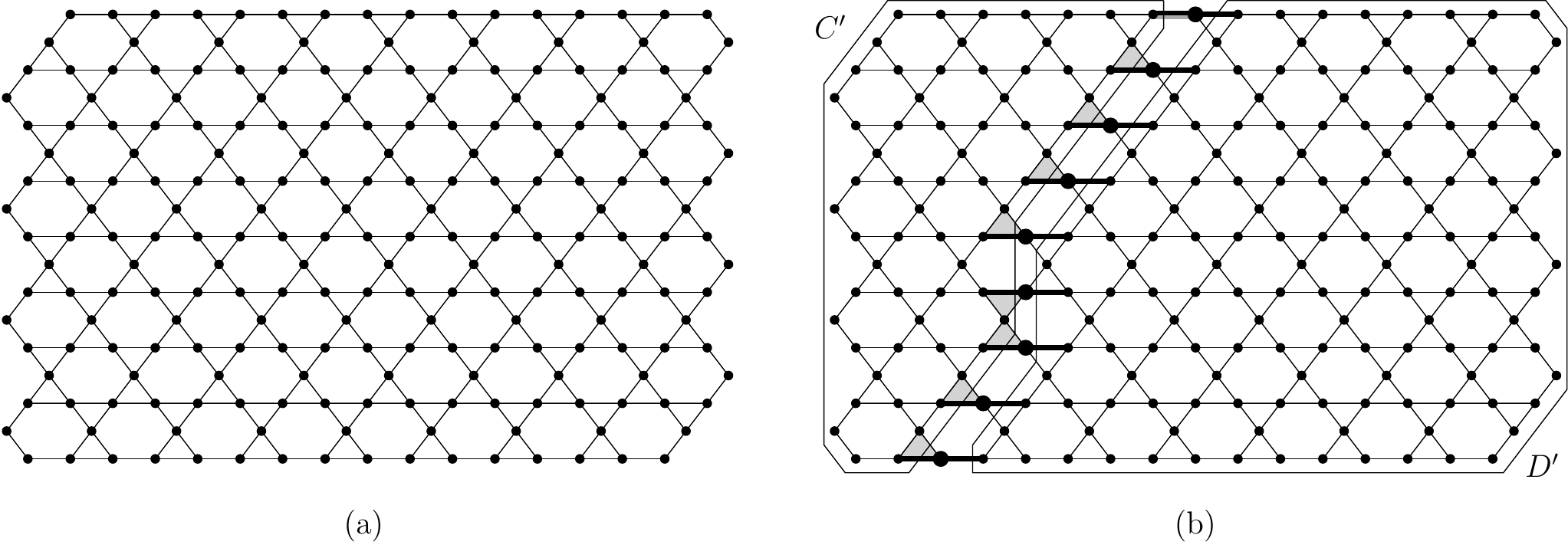}
\end{center}
\caption{(a) $L(W_8)$, the line graph of an elementary wall of height $8$. (b) The set of nine vertices depicted with large black disks is a minimal separator $S_{(1,1,0,0,1,1,1,1)}'$ in $L(W_8)$, which corresponds to the minimal separator $S_{(1,1,0,0,1,1,1,1)}$ in $W_8$. The two
components
of $L(W_8)-S_{(1,1,0,0,1,1,1,1)}'$ are also depicted.}\label{fig:line-wall}
\end{figure}

\begin{repproposition}{proposition:line-walls}
\label{prop:line-walls}
For every even integer $h\ge 2$, the graph $L(W_{h})$ has at least $2^{h/2}$ minimal separators.
\end{repproposition}

\begin{proof}
We use a modification of the construction used in the proof of Proposition~\ref{prop:walls}. We again consider the minimal separators $S_x$ in $W_h$ constructed in the proof of Proposition~\ref{prop:walls}; however, for technical reasons that will simplify the argument, we restrict ourselves only to the $2^{h/2}$ minimal separators $S_x$ in $W_h$ that arise
from binary sequences $x\in X_h$, where $$X_h = \{(x_1,\ldots, x_h)\in \{0,1\}^h\mid x_{2i-1} = x_{2i}\textrm{ for all }i\in \{1,2,\ldots, h/2\}\}\,.$$
Recall that for every $x\in X_h$, we have $S_x = \{v^{x,0}, v^{x,1}, \ldots, v^{x,h}\}$ where $v^{x,0} = (2,0)$ and $v^{x,j} = v^{x,j-1}+(x_j,1)$ for all $j \in \{1,\ldots, h\}$. A set of $2^{h/2}$ minimal separators of $L(W_h)$ can be obtained as follows. For each $x\in X_h$, we define a set $S_x'\subseteq V(L(W_h))$ as follows: $S_x' = \{e^{x,j}\mid v^{x,j}\in S_x\}$ where $e^{x,j}$ is the vertex of the line graph of $W_h$ corresponding to the edge in $W_h$ joining vertex $v^{x,j}$ with vertex $v^{x,j}+(1,0)$.

Since the mapping is clearly one-to-one, the set $\{S_x'\mid x\in X_h\}$ is of cardinality $2^{h/2}$. Therefore, to complete the proof it suffices to show that for every $x\in X_h$, set $S_x'$ is a minimal separator in $L(W_h)$. Let us first argue that the graph $L(W_h)-S_x'$ is disconnected.
Vertices of the wall $W_h$ correspond bijectively to maximal cliques of its line graph. For every $x\in X_h$, every vertex of the form $v^{x,j}$ where $j\in \{1,\ldots, h-1\}$ corresponds to a triangle (clique of size three) in $L(W_h)$, while vertex $v^{x,h}$ corresponds to a clique of size two.
Let us say that a triangle in $L(W_h)$ is \emph{upward pointing} if it arises from a vertex in $W_h$ whose coordinates have even sum, and \emph{downward pointing}, otherwise. (We draw this terminology from the planar embeddings of the line graphs of the walls following the example given in Fig.~\ref{fig:line-wall}.) It is not difficult to see that for every $x\in X_h$ and every even $i\in \{0,1,\ldots, h-2\}$, vertex $v^{x,i}$ corresponds to an upward triangle, while odd-indexed vertices may correspond to either upward or downward pointing triangles. It follows that for no index $i\in \{0,1,\ldots, h-1\}$,
vertices $v^{x,i}$ and $v^{x,i+1}$ can both correspond to downward pointing triangles. This property ensures that the graph $L(W_h)-S_x'$ is disconnected, with exactly two components $C'$ and $D'$ such that for all $v^{x,j}\in S_x$, component $C'$ contains all vertices of the form $e^{x,j^-}$, where $e^{x,j^-}\in V(L(W_h))$ is the vertex corresponding to the edge in $W_h$ joining vertex $v^{x,j}$ with vertex $v^{x,j}-(1,0)$,
while component $D'$ contains all vertices of the form $e^{x,j^+}$,
where $e^{x,j^+}\in V(L(W_h))$ is the vertex corresponding to the edge in $W_h$ joining vertex $v^{x,j}+(1,0)$ with vertex $v^{x,j}+(2,0)$.
Furthermore, since for every vertex
$e^{x,j}\in S_x'$, vertices $e^{x,j^-}$ and $e^{x,j^+}$ are both adjacent to $e^{x,j}$ in $L(W_h)$, this also implies, by Lemma~\ref{lem:characterization-minimal-separators}, that $S_x'$ is a minimal separator in $L(W_h)$. This completes the proof.
Fig.~\ref{fig:line-wall}(b) shows an example with $h = 8$ and $x = (1,1,0,0,1,1,1,1)$. The thick horizontal edges can be used to justify the fact that $C'$ and $D'$ are $S_x'$-full components of $L(W_h)-S_x'$.
\qed\end{proof}

Since line graphs of elementary walls are $\{$claw, $K_4$, $C_4$, diamond$\}$-free, Proposition~\ref{prop:line-walls} implies the following.

\begin{corollary}
\label{cor:wallcontainedinf}
The class of $\{$claw, $K_4, C_4$, diamond$\}$-free graphs is not tame.
\end{corollary}

\section{Graph operations} \label{sec:operations}

We now study the effect of various graph operations on the number of minimal separators. The family of minimal separators of a disconnected graph can be computed from the families of minimal separators of its components, and a similar statement holds for graphs whose complements are disconnected. The correspondences are as follows, see~\cite[Theorem~3.1]{MR2971360}.

\begin{reptheorem}{theorem:joinunion}
\label{theorem:joinunion}
If $G$ is a disconnected graph, with components $G_1,\ldots, G_k$, then \hbox{$\Sep_G=\{\emptyset\} \cup \bigcup_{i=1}^k \Sep_{G_i}$}. If $G$ is the join of graphs $G_1,\ldots, G_k$, then
$S\in \Sep_G$ if and only if there exists some $i\in \{1,\ldots, k\}$ and some $S_i\in \Sep_{G_i}$ such that $S=S_i\cup (V(G)\setminus V(G_i))$.
\end{reptheorem}

In order to prove Corollary~\ref{cor:tame-disconnected}, we first derive the following corollaries expressing the 
number of minimal separators of a disconnected graph in terms of the 
numbers of minimal separators of its components, and similarly in the case when the complement of $G$ is disconnected.

\begin{repcorollary}{corollary:disconnected}
\label{cor:disconnected}
Let $G$ be a disconnected graph, with components $G_1,\ldots, G_k$. Then $s(G) = \sum_{i = 1}^ks(G_i)+1$.
\end{repcorollary}

\begin{proof}
Immediate from the first statement of Theorem~\ref{theorem:joinunion} and the fact that sets $\{\emptyset\}, \Sep_{G_1},\ldots, \Sep_{G_k}$, $i\in \{1,\dots, k\}$ are pairwise disjoint.
\qed\end{proof}

\begin{sloppypar}
\begin{corollary}\label{cor:co-disconnected}
Let $G$ be a graph with disconnected complement, with co-components $G_1,\ldots, G_k$. Then $s(G) = \sum_{i = 1}^ks(G_i)$.
\end{corollary}
\end{sloppypar}

\begin{proof}
Immediate from the second statement of Theorem~\ref{theorem:joinunion}
and the fact that sets $X_i = \{S_i\cup (V(G)\setminus V(G_i))\mid S_i\in \Sep(G_i)\}$ for $i \in \{1,\ldots, k\}$ are pairwise disjoint, with $|X_i| = s(G_i)$.
\qed\end{proof}

\begin{repcorollary}{corollary:tame-disconnected}
\label{cor:tame-disconnected}
Let $\mathcal{G}$ be a hereditary class of graphs and let $\mathcal{G}'$ be the class of connected graphs in $\mathcal{G}$. Then $\mathcal{G}$ is tame if and only if $\mathcal{G}'$ is tame.
\end{repcorollary}

\begin{proof}
Since $\mathcal{G}$ is hereditary, we have $\mathcal{G}'\subseteq \mathcal{G}$. Hence, if $\mathcal{G}$ is tame, then so is $\mathcal{G}'$.
Suppose that $\mathcal{G}'$ is tame. By Lemma~\ref{lem:poly-monomial}, there exists a least integer $k\ge 2$ such that $s(G)\le |V(G)|^k$ for all $G\in \mathcal{G}'$. Let $G\in \mathcal{G}\setminus \mathcal{G'}$ and let $G_1,\ldots, G_p$ (with $p\ge 2$) be the
components of $G$. Since $G$ is disconnected and for all $i\in \{1,\ldots, p\}$ we have $G_i\in \mathcal{G}'$, we infer using Corollary~\ref{cor:disconnected} that
$s(G)= \sum_{i = 1}^p s(G_i)+1\le \sum_{i = 1}^p |V(G_i)|^k+1\le
\left(\sum_{i = 1}^p |V(G_i)|\right)^k = |V(G)|^k$. The first
inequality holds by the assumption on $\mathcal{G}'$ and the second one since $k\ge 2$. It follows that $\mathcal{G}$ is tame. \qed
\end{proof}

\begin{repcorollary}{corollary:tame-co-connected}
\label{cor:tame-co-connected}
Let $\mathcal{G}$ be a hereditary class of graphs and let $\mathcal{G}'$ be the class of co-connected graphs in $\mathcal{G}$. Then $\mathcal{G}$ is tame if and only if $\mathcal{G}'$ is tame.
\end{repcorollary}

\begin{proof}
Since $\mathcal{G}$ is hereditary, we have $\mathcal{G}'\subseteq \mathcal{G}$. Thus, if $\mathcal{G}$ is tame, then so is $\mathcal{G}'$.
Suppose that $\mathcal{G}'$ is tame. By Lemma~\ref{lem:poly-monomial}, there exists a least integer $k\ge 1$ such that $s(G)\le |V(G)|^k$ for all $G\in \mathcal{G}'$. Let $G\in \mathcal{G}\setminus \mathcal{G'}$ and let $G_1,\ldots, G_p$ (with $p\ge 2$) be the co-components of $G$.
Since for all $i\in \{1,\ldots, p\}$ we have $G_i\in \mathcal{G}'$, we infer using Corollary~\ref{cor:co-disconnected} and the assumption on $\mathcal{G}'$ that
$s(G)=\sum_{i = 1}^p s(G_i)\le \sum_{i = 1}^p |V(G_i)|^k\le
\left(\sum_{i = 1}^p |V(G_i)|\right)^k = |V(G)|^k$. The first
inequality holds by the assumption on $\mathcal{G}'$ and the second one since $k\ge 1$. It follows that $\mathcal{G}$ is tame. \qed
\end{proof}

McKee observed in~\cite{MR3743965} that if $G_1$ is an induced subgraph of
$G_2$, then every minimal separator of $G_1$ is contained in a minimal separator in $G_2$. The proof actually shows that the following monotonicity property holds.

\begin{repproposition}{proposition:monotonicity}
\label{prop:monotonicity}
If $G_1$ is an induced subgraph of $G_2$, then $s(G_1)\le s(G_2)$.
\end{repproposition}

\begin{proof}
    We define a one-to-one mapping $\phi$ from the set $\Sep_{G_1}$ to the set $\Sep_{G_2}$. For every minimal separator $S$ in $G_1$, we construct $\phi(S)$ as follows.
    Since $S$ is a minimal separator of $G_1$, there exist two $S$-full components of $G-S$, say $C$ and $D$. Clearly, the set $S\cup (V(G_2)\setminus V(G_1))$ separates $C$ and $D$ in $G_2$.
    Let $\phi(S)$ be any subset of $S\cup (V(G_2)\setminus V(G_1))$ separating $C$ and $D$ that is inclusion-minimal with this property. (Such a set can be obtained, for example, by iteratively deleting vertices from $S\cup (V(G_2)\setminus V(G_1))$ if the resulting set still separates $C$ from $D$.) By construction, $C$ and $D$ are $\phi(S)$-full components of $G_2$, hence $\phi(S)$ is a minimal separator in $G_2$.
Moreover, we clearly have $S\subseteq \phi(S)$ and since by construction
    $\phi(S)\cap V(G_1)\subseteq S$, it follows that $S = \phi(S)\cap V(G_1)$.
Thus, if $\phi(S) = \phi(S')$ for $S,S'\in \Sep_{G_1}$, then $S = S'$, and mapping $\phi$ is one-to-one. Hence $s(G_1)\le s(G_2)$, as claimed. \qed
\end{proof}

In view of Proposition~\ref{prop:monotonicity}, it is natural to ask how large can the gap $s(G_2)-s(G_1)$ be if the graphs $G_1$ and $G_2$ are not ``too different'', for example, if $G_1$ is obtained from $G_2$ by deleting only one vertex. In the following three propositions we identify three  properties of a vertex $v$ in a graph $G$ such that deleting $v$ either leaves the minimal separators unchanged or decreeases it by one.

\begin{proposition}\label{prop:universalvertex}
Let $G$ be a graph with at least two vertices and let $v$ be a universal vertex in $G$. Then $s(G) = s(G-v)$.
\end{proposition}
\begin{proof}
Immediate from Corollary~\ref{cor:co-disconnected}, using the fact that $G$ is isomorphic to the join of $G-v$ and $P_1$, and that $s(P_1)=0$. 
\end{proof}

To prove Proposition~\ref{prop:truetwin}, the following lemma will be useful.

\begin{replemma}{lemmaa:sameneighbors}
\label{lemma:sameneighbors}
Let $G$ be a graph and $v,w$ vertices in $G$ such that $N_G(v)\setminus \{w\}=N_G(w)\setminus\{v\}$. If $S$ is a minimal separator in $G$, then
$v\in S$ if and only if $w\in S$.
\end{replemma}

\begin{proof}
Let $S\in \Sep_G$. By symmetry, it suffices to show that $v\in S$ implies $w\in S$. Suppose for a contradiction that $v\in S$ but $w\not\in S$.
By Lemma~\ref{lem:characterization-minimal-separators}, the graph $G-S$ has two $S$-full components, say $C$ and $D$.
We may assume without loss of generality that $w\not\in V(C)$.
Since $C$ is an $S$-full component of $G-S$,
vertex $v\in S$ has a neighbor $y$ in $V(C)$.
But now, $y$ is a vertex contained in $N_G(v)\setminus \{w\}$
but not in $N_G(w)\setminus\{v\}$, a contradiction. \qed
\end{proof}

\begin{repproposition}{proposition:truetwin}
\label{prop:truetwin}
Let $G$ be a graph having a pair of true twins $v,w$ with $v\neq w$.
Then $s(G)= s(G-v)$.
\end{repproposition}

\begin{proof}
Let $G' = G-v$. From Proposition~\ref{prop:monotonicity} it follows that $s(G)\ge s(G-v)$, so we have to prove that $s(G)\le s(G-v)$.
We will prove
it by constructing a one-to-one mapping $\phi$ from $\Sep_{G}$ to $\Sep_{G'}$. The function is defined by the following rule: for every $S\in \Sep_G$, we set $\phi(S)=S\setminus \{v\}$.

We first show that $\phi$ maps minimal separators in $G$ to minimal separators in $G'$. Let $S\in \Sep_G$. If $v\in S$, then $\phi(S) = S\setminus \{v\}$, which is a minimal separator in $G'$ by Corollary~\ref{cor:deleting-a-vertex-in-S}. Suppose now that $v\not\in S$. Then $\phi(S) = S$ and $w\not\in S$ by Lemma~\ref{lemma:sameneighbors}. Let $C$ and $D$ be two $S$-full components of $G-S$ and let $K$ be the component of $G-S$ containing $v$.
If $K\not\in \{C,D\}$, then $C$ and $D$ are two $S$-full components of
$G'-S$ and hence $\phi(S) = S\in \Sep_{G'}$ in this case.
We may thus assume that $K = C$. Note that $w\in V(C)$ since vertices $v$ and $w$ are adjacent in $G$. Moreover, since $v$ and $w$ are true twins in $G$, they are also true twins in $C$. This implies that the graph $C-v$ is connected and hence a component of $G'-S$. Note that $D$ is an $S$-full component of $G'-S$. We complete the proof that $S$ is a minimal separator in $G'$ by showing that $C-v$ is also an $S$-full component of $G'-S$. Consider an arbitrary vertex $x\in S$. Since $C$ is an $S$-full component of $G-S$, vertex $x$ has a neighbor  in $C$. However, since $v$ and $w$ are true twins in $G$, it cannot be that $N_G(x)\cap V(C) = \{v\}$. It follows that $x$ also has a neighbor in $C-v$. Therefore, $C-v$ is an $S$-full component of $G'-S$, as claimed.

It remains to show that $\phi$ is one-to-one. Let $S_1$ and $S_2$ be distinct minimal separators in $G$. If $S_1\cap \{v\}\neq S_2\cap \{v\}$, then we have $\phi(S_1)\cap \{w\}\neq \phi(S_1)\cap \{w\}$, inferring that $\phi(S_1)\neq \phi(S_2)$.
If none of the sets $S_1$, $S_2$ contains $v$, then $\phi_1(S_1)=S_1$ and $\phi(S_2)=S_2$. If both of the sets $S_1$ and $S_2$ contain $v$, then $\phi(S_1)=S_1\setminus \{v\}\neq S_2\setminus \{v\}=\phi(S_2)$.
It follows that $\phi$ is one-to-one, as claimed.
\qed\end{proof}

\begin{repproposition}{proposition:simplicialvertex}
\label{prop:simplicialvertex}
Let $G$ be a graph and let $v$ be a simplicial vertex in $G$.
Then $s(G-v)\le s(G)\le s(G-v)+1\,.$
\end{repproposition}

\begin{proof}
Let $K=N_G(v)$ and $G' = G-v$. The inequality $s(G-v)\le s(G)$ follows directly from Proposition~\ref{prop:monotonicity}. We will prove the second inequality by showing the inclusion $\Sep_{G} \subseteq \Sep_{G'}\cup\{K\}$.

Let $S\in \Sep_{G}$. If $S = K$, then $S\in \Sep_{G'}\cup\{K\}$, so we may assume that $S\neq K$. From Lemma~\ref{lem:characterization-minimal-separators} it follows that the graph $G-S$ has two $S$-full components $C$ and $D$. Since $V(C)$ and $V(D)$ are anticomplete to each other in $G$ and since $N_G(v)$ is a clique, it follows that $v\notin S$. Let $C_v$ be the component of $G-S$ containing $v$. If $C_v\not\in \{C,D\}$, then $C$ and $D$ are two $S$-full components of
$G'-S$ and hence $S\in \Sep_{G'}$. We may thus assume that $C_v = C$.
If $V(C)=\{v\}$, then $K\subseteq S$ and since $C$ is an $S$-full component of $G-S$, it follows that $S=K$, a contradiction. So we have that $|V(C)|\ge 2$; moreover, since $v$ is a simplicial vertex in a connected graph $C$, the graph $C-v$ is connected. Note that $D$ is an $S$-full component of $G'-S$. We complete the proof that $S$ is a minimal separator in $G'$ by showing that $C-v$ is also an $S$-full component of $G'-S$. Suppose for a contradiction that this is not the case, that is, there exists a vertex $x\in S$ without a neighbor in $C-v$. Since $C$ is an $S$-full component of $G-S$, vertex $x$ has a neighbor in $C$. Hence, $N_G(x)\cap V(C) = \{v\}$. However, since $C$ is a connected graph containing $v$, it follows that $N_G(v)\cap V(C)\neq \emptyset$. Taking an arbitrary $w\in N_G(v)\cap V(C)$, we now obtain that $x$ and $w$ are a pair of non-adjacent neighbors of $v$ in $G$, contradicting the fact that $v$ is a simplicial vertex in $G$. Therefore, $C-v$ is an $S$-full component of $G'-S$, as claimed, and hence $S$ is a minimal separator in $G'$; in particular, $S\in \Sep_{G'}\cup\{K\}$.
Since $S\in \Sep_G$ was arbitrary, this shows $\Sep_G\subseteq \Sep_{G'}\cup\{K\}$. \qed
\end{proof}

\section{Proof of Theorems~\ref{thm:main} and \ref{thm:main-dual}} \label{sec:proof-of-main-theorem}

In this section we prove Theorems~\ref{thm:main} and \ref{thm:main-dual}. We do this in several steps. We start with a proposition giving a necessary condition for a family $\F$ of graphs so that the class of $\F$-free graphs is tame.

\begin{repproposition}{proposition:forestcycle}
\label{prop:forestcycle}
Let $\F$ be a finite family of graphs such that
for every $F\in \F$ we have $F\nsubseteq_iP_4$, $F\nsubseteq_i 2\2$.
If, in addition, all graphs in $\F$ contain cycles or all of them are of girth more than $5$, then the class of $\F$-free graphs is not tame.
\end{repproposition}

\begin{proof}
We analyze the two cases separately.

\textit{Case 1: all graphs in $\F$ contain cycles.} Let $\ell$ be the smallest integer such that $\ell \ge 3$ and for every graph $F\in\F$, it holds that $F$ does not contain an induced cycle of length exactly $2\ell$. Note that $\ell $ is well-defined since $\F$ is finite. We claim that every graph in $\Theta_{\ell}$ is $\F$-free. Suppose for a contradiction that for some $k\ge 2$, the graph $\theta_{k,\ell}$ contains an induced subgraph isomorphic to some $F\in \F$. Since $F$ contains an induced cycle and every induced cycle contained in $\theta_{k,\ell}$ is of length $2\ell$, we infer that $F$ contains an induced cycle of length $2\ell$. However, this contradicts the definition of $\ell$. Thus, every graph in $\Theta_{\ell}$ is $\F$-free, as claimed. By Corollary~\ref{cor:theta}, the class of $\F$-free graphs is not tame.

\textit{Case 2: every graph in $\F$ is of girth more than five.}
We will show that in this case, every graph in ${\cal L}_2$ is $\F$-free.
By Corollary~\ref{cor:line-theta} this will imply that the class of
$\F$-free graphs is not tame. From the definition of ${\cal L}_2$ it follows that every graph in ${\cal L}_2$ has independence number two. Thus, to show that every graph in ${\cal L}_2$ is $\F$-free, it suffices to prove that $\alpha(F)\ge 3$ for all $F\in \F$. Suppose for a contradiction that $\alpha(F)\le 2$ for some $F\in \F$. Then $F$ is acyclic, since otherwise a shortest cycle in $F$ would be of length at least $6$, which would imply $\alpha(F)\ge 3$. Moreover, $F$ has at most two connected components.
If $F$ is connected, then $F$ is a tree with $\alpha(F)\le 2$.
In particular, the maximum degree of $F$ is at most $2$, hence
$F$ a path with at most four vertices, which implies $F\subseteq_i P_4$, a contradiction. If $F$ has exactly two components, then the condition $\alpha(F)\le 2$ implies that each component of $F$ is a complete graph. However, since $F$ is acyclic, each component of $F$ is induced subgraph of $P_2$. It follows that $F\subseteq_i 2P_2$, a contradiction. \qed
\end{proof}

The following sufficient condition is derived using Ramsey's theorem.

\begin{proposition}\label{prop:Ramsey}
For every two positive integers $k$ and $\ell$, the class of
$\{\2+kP_1,K_\ell+\2\}$-free graphs is tame.
\end{proposition}

\begin{proof}
We may assume that $k\ge 2$ and $\ell\ge 2$. Then $R(\ell,k)\ge 2$.
Let $G$ be a $\{\2+k\1,K_\ell+\2\}$-free graph. We will prove that for every minimal separator $S$ in $G$, there exists a set $X\subseteq V(G)$ such that $|X|\le R(\ell,k)-1$ and $S = N(X)$. Clearly, this will imply that $G$ has at most ${|V(G)|\choose  R(\ell,k)-1}$ minimal separators. Let $S$ be a minimal separator in $G$ and let $C$ and $D$ be two $S$-full components of $G-S$. Since $N(V(C))  = N(V(D)) = S$, it suffices to show that $|V(C)|\le R(\ell,k)-1$ or $|V(D)|\le R(\ell,k)-1$. Suppose that this is not the case. Then $|V(C)|\ge R(\ell,k)$ and $|V(D)|\ge R(\ell,k)$.
By Ramsey's theorem, this implies that there exists a set $Z\subseteq V(C)$ such that $Z$ is either a clique of size $\ell$ or an independent set of size $k$. But then $Z$ together with a pair of adjacent vertices from $D$
induces either a $K_\ell+\2$ or $\2+kP_1$, respectively.
Both cases lead to a contradiction.
\qed\end{proof}

The next proposition simplifies the cases with $P_3+\1\in \F$. 

\begin{repproposition}{proposition:3p1andF-free}
\label{prop:3p1andF-free}
Let $\F$ be a family of graphs such that
$P_3+\1\in \F$ and let
\hbox{$\F' = (\F\setminus \{P_3+\1\})\cup \{3\1\}$}.
Then the class of $\F$-free graphs is tame if and only if the class of $\F'$-free graphs is tame.
\end{repproposition}

\begin{proof}
Let ${\cal G}$ and ${\cal G}'$ be the classes of $\F$-free and $\F'$-free graphs, respectively. Since $\F'\trianglelefteq \F$, we have that $\mathcal{G}'\subseteq \mathcal{G}$. Hence, if $\mathcal{G}$ is tame, then so is $\mathcal{G}'$. Suppose that $\mathcal{G}'$ is tame. By Lemma~\ref{lem:poly-monomial}, there exists a least integer $k\ge 0$ such that $s(G)\le |V(G)|^k$ for all $G\in \mathcal{G}'$. Let $G\in \mathcal{G}$. By Corollary~\ref{cor:tame-co-connected} we may assume that $G$ is co-connected. Since $G$ is co-connected and $P_3+\1$-free, Theorem~\ref{thm:pawfree} implies that $G$ is either a disjoint union of complete graphs, in which case $s(G)\le 1$, or $G$ is $3\1$-free, in which case $G\in  \mathcal{G}'$ and thus $s(G)\le |V(G)|^k$. Thus, in both cases we have that $s(G)\le |V(G)|^k$. It follows that ${\cal G}$ is tame. \qed
\end{proof}

We now consider various families of forbidden induced subgraphs with at most four vertices. Propositions~\ref{prop:Ramsey} and~\ref{prop:3p1andF-free} can be used to prove the following.

\begin{repproposition}{proposition:k3+p1free}
\label{prop:k3+p1free}
For every $F\in \{4\1$, $\2+2\1$, $P_3+\1$, claw$\}$, the class of $\{F,K_3+\1\}$-free graphs is tame.
\end{repproposition}

\begin{proof}
\begin{enumerate}[i)]
\item The class of $\{4\1,K_3+\1\}$-free graphs is
a subclass of the class of $\{\2+4\1,K_3+\2\}$-free graphs, which is tame by Proposition~\ref{prop:Ramsey}.

\item The class of $\{\2+2\1,K_3+\1\}$-free graphs is
a subclass of the class of $\{\2+2\1,K_3+\2\}$-free graphs, which is tame by Proposition~\ref{prop:Ramsey}.

\item By Proposition~\ref{prop:3p1andF-free}, it suffices to show that the class of $\{3\1,K_3+\1\}$-free graphs is tame. This follows from part i) of the proposition. \qed
\end{enumerate}
\end{proof}

To prove Proposition~\ref{prop:C4p3+p1}, we will need the following result about the structure of $\{3\1,C_4\}$-free graphs.

\begin{theorem}[Choudum and Shalu~\cite{MR2139768}]\label{thm:3k1c4}
If $G$ is a $\{3\1,C_4\}$-free graph, then either $G$ is chordal or
$G$ is isomorphic to the graph \mbox{$C_5(m_1,m_2,m_3,m_4,m_5)\ast K_t$} for some integers $m_i\ge 1$ and $t\ge 0$, where $C_5(m_1,m_2,m_3,m_4,m_5)$ is the graph obtained from $C_5$ with a cyclic order of vertices $v_1,\ldots, v_5$ by replacing each $v_i$ by $K_{m_i}$ and joining every pair of vertices $x\in K_{m_i}$ and $y\in K_{m_j}$ if and only if $v_i$ and $v_j$ are adjacent in the $C_5$.
\end{theorem}

The next result follows from a structural property of
$\{3\1,C_4\}$-free graphs proved by Choudum and Shalu~\cite{MR2139768}.

\begin{repproposition}{proposition:C4p3+p1}
\label{prop:C4p3+p1}
The class of $\{P_3+\1,C_4\}$-free graphs is tame.
\end{repproposition}

\begin{proof}
Let $G$ be a $\{3\1, C_4\}$-free graph. We will prove that $s(G)\le |V(G)|$.
By Propositions~\ref{prop:universalvertex} and~\ref{prop:truetwin}, we may assume that $G$ has no universal vertices and is true-twin-free.
From Theorem~\ref{thm:3k1c4} it follows that $G$ is either chordal or isomorphic to $C_5$. Hence, using the fact that chordal graphs have less than $|V(G)|$ minimal separators~\cite{MR0408312} and the fact that $s(C_5) = 5$, we infer that $G$ has at most $|V(G)|$ minimal separators. 
\qed
\end{proof}

The next two propositions are proved using a structural analysis of the graphs in respective classes.

\begin{repproposition}{proposition:C4p2+2p1}
\label{prop:C4p2+2p1}
The class of $\{\2+2\1,C_4\}$-free graphs is tame.
\end{repproposition}

\begin{proof}
Let $G$ be a $\{P_2+2\1,C_4\}$-free graph. By Propositions~\ref{prop:universalvertex} and~\ref{prop:truetwin}, we may assume that $G$ has no universal vertices and is true-twin-free.
Let $I$ be a maximum independent set in $G$. If $|I|\le 2$, then $G$ is $\{3\1, C_4\}$-free and  Proposition~\ref{prop:C4p3+p1}
implies that $G$ has a polynomially bounded number of minimal separators in this case.
Thus, in what follows we assume that $|I|\ge 3$. The maximality of $I$ implies that every vertex in $V(G)\setminus I$ has a neighbor in $I$. Let $v\in V(G)\setminus I$ be arbitrary. If $v$ has at least two non-neighbors in $I$, then we have an induced $\2+2\1$ in $G$. It follows that the vertex set of $G$ can be partitioned into three pairwise disjoint sets, $V(G)=I\cup X\cup Y$, where $X=\{v\in V(G)\setminus I \mid |I\setminus N(v)|=1\}$ and $Y=\{v\in V(G)\setminus I  \mid I\subseteq N(v)\}$. Note that every two distinct vertices $x\in X\cup Y$ and $y\in Y$ have two common neighbors in $I$. Since $I$ is independent and $G$ is $C_4$-free, we infer that $X$ and $Y$ are complete to each other
and $Y$ is a clique. It follows that vertices in $Y$ are universal in $G$.
However, $G$ has no universal vertices, which implies $Y = \emptyset$.

Suppose that $|I| \ge 4$. Then any two vertices in $X$ have two common neighbors in $I$ and are therefore adjacent to each other, since otherwise $G$ would contain an induced $C_4$. It follows that $X$ is a clique and so $G$ is a split graph, with a split partition $(X,I)$. Consequently, we infer
that $G$ is chordal and hence has less than $|V(G)|$ minimal separators~\cite{MR0408312}.
 
It remains to consider the case that $|I| = 3$. We consider two cases depending on whether $G$ contains an induced $C_6$ or not.
\medskip

\textit{Case 1: $C_6\subseteq_i G$.}


Let $C$ be an induced $6$-cycle in $G$, with a cyclic order of vertices $v_1,\dots,v_6$. Throughout the proof of this case, we consider indices modulo $6$. We now analyze the neighborhood of an arbitrary vertex $v\in V(G)\setminus V(C)$ in $C$. The sets $I_1 = \{v_1,v_3,v_5\}$ and $I_2 = \{v_2,v_4,v_6\}$ are maximum independent sets in $G$ and hence the same arguments as we used above for set $I$ imply that
for $j\in \{1,2\}$, vertex $v$ has exactly two neighbors in $I_j$.
By symmetry, we may assume that $v$ is adjacent to $v_1$ and $v_3$ but not to $v_5$. Since $G$ is $C_4$-free, we infer that $v$ is adjacent to $v_2$.
But now, using the symmetry and the fact that $v$ has exactly two neighbors in $\{v_2,v_4,v_6\}$, we may assume that $v$ is adjacent to $v_4$ but not to $v_6$.
It follows that the only possibility for an arbitrary vertex $v\in V(G)\setminus V(C)$ is that $N(v)\cap V(C)$ contains exactly four consecutive vertices on the cycle. Thus, setting for all $i\in \{1,\dots,6\}$
$$A_i=\{v\in V(G)\setminus V(C)\mid N(v)\cap V(C)=\{v_i,v_{i+1},v_{i+2}, v_{i+3}\}\}\,,$$
we infer that the vertex set of $G$ can be partitioned as $V(G)= V(C)\cup A_1\cup\ldots \cup A_6$. Moreover, every set $A_i$ is a clique, since otherwise two non-adjacent vertices in $A_i$ together with vertices $v_i, v_{i+2}$ would form an induced $C_4$ in $G$. Next, we examine the adjacencies between the sets $A_i$ and $A_j$ for $i,j\in \{1,\dots,6\}$, $i\neq j$, in $G$.
Fix $i\in \{1,\ldots, 6\}$ and let $v\in A_i$ and $w\in A_{i+1}$ be arbitrary. If $v$ is adjacent to $x$, then there is an induced $C_4$ with vertex set $\{v,v_{i}, w, v_{i+3}\}$ in $G$. It follows that any two cliques $A_i$ and $A_{i+1}$ are complete to each other in $G$. Now let $x\in A_{i+2}$ be arbitrary. If $v$ is non-adjacent to $z$, then $G$ contains an induced $C_4$ with vertex set $\{v,v_i,v_{i+5},x\}$. This implies that $A_i$ and $A_{i+2}$ are anticomplete to each other in $G$. Finally, let $y\in A_{i+3}$. If $v$ is not adjacent to $y$, then $G$ contains an induced $C_4$ with vertex set $\{v,v_i,y,v_{i+3}\}$. Thus any two cliques $A_i$ and $A_{i+3}$ are complete to each other in $G$. We have thus showed that for every $i\in \{1,\dots,6\}$, any two vertices in $A_i$ are true twins in $G$. Since $G$ is true-twin-free, this implies that $|A_i|\le 1$ for all $i\in \{1,\dots,6\}$ and consequently $G$ has at most $12$ vertices.
Thus, $G$ has at most a constant number of minimal separators in this case.

\medskip
\textit{Case 2: $G$ is $C_6$-free.}

Recall that $I$ is a maximum independent set in $G$ with $|I| = 3$ and
$|I\setminus N(v)|=1$ for all $v\in V(G)\setminus I$. Let $I = \{v_1,v_2,v_3\}$ and let $X_j = \{v\in V(G)\setminus I\mid N(v)\cap I = I\setminus\{v_j\}\}$ for all $j\in \{1,2,3\}$. Then, for every $j\in \{1,2,3\}$, we have that $X_j$ is a clique in $G$, since otherwise we have an induced $C_4$ in $G$. We claim that for each $i\in \{1,2,3\}$, the graph $G-v_i$ is $3\1$-free. By symmetry it suffices to consider the case $i = 1$. Suppose for a contradiction that $G-v_1$ is not $3\1$-free and let $J$ be an independent set of size three in $G-v_1$. If $v_2\in J$, then $J\setminus \{v_2\}$ is a subset of the clique $X_1\cup\{v_3\}$, which implies $|J|\le 2$. Hence $v_2\not\in J$ and by symmetry $v_3\not\in J$.
Since all sets $X_j$ for $j\in \{1,2,3\}$ are cliques, we infer that $J = \{x_1,x_2,x_3\}$ with $x_j\in X_j$ for all $j\in \{1,2,3\}$. But then $G$ contains an induced $C_6$ on the vertex set $\{v_1,x_2,v_3,x_1,v_2,x_3\}$, a contradiction. This shows that for each $i\in \{1,2,3\}$, the graph $G-v_i$ is $3\1$-free and hence $\{3\1,C_4\}$-free. By Proposition~\ref{prop:C4p3+p1} and its proof, $G-v_i$ has at most $|V(G-v_i)|= |V(G)|-1$ minimal separators.

Note that
$$\Sep_G=\left(\bigcup_{1\le i<j\le 3}\Sep_G(v_i,v_j)\right)
 \cup \left(\bigcup_{i = 1}^3\bigcup_{x\in X_i}\Sep_G(v_i,x)\right)
 \cup \left(\bigcup_{1\le i<j\le 3}\bigcup_{\substack{x\in X_i,y\in X_j\\xy\not\in E(G)}}\Sep_G(x,y)\right)\,,$$
since these are the only possibilities of choosing two non-adjacent vertices in $G$. We will now bound the cardinalities of the sets or the right side of the above equality.

First, let $i,j\in \{1,2,3\}$ such that $i<j$ and let $k\in \{1,2,3\}\setminus\{i,j\}$. It is not difficult to see that if $S\in
\Sep_G(v_i,v_j)$, then $S = X_k\cup Z$, where
$Z\in \{X_i, X_j\}\cup\{S\cup \{v_k\}\mid S\in \Sep_{G[X_i\cup X_j]}\}$.
Since the graph $G[X_i\cup X_j]$ is an induced subgraph of $G-v_k$,
we obtain, using Proposition~\ref{prop:monotonicity},
that $s(G[X_i\cup X_j])\le s(G-v_k)\le |V(G)|-1$.
This implies $|\Sep_G(v_i,v_j)|\le s(G[X_i\cup X_j])+2\le s(G-v_k)+2\le |V(G)|+1$.

\begin{sloppypar}
Next, we consider sets of the form $\Sep_G(v_k,x)$ where $x\in X_k$.
For all $S\in \Sep_{G-\{v_i,v_j\}}(v_k,x)$, we let
$f(S)= \{q\in \{i.j\} \mid X_q\nsubseteq S\}$ and
$\phi(S)=S\cup \{v_q\mid q\in f(S)\}$.
It is not difficult to see that $\phi$ is a function from $\Sep_{G-\{v_i,v_j\}}(v_k,x)$ to $\Sep_G(v_k,x)$ such that for each $Z\subseteq \{i,j\}$, function $\phi$ maps the set $\{S\in \Sep_{G-\{v_i,v_j\}}(v_k,x)\mid f(S)= Z\}$ bijectively to the set
\hbox{$\{S\in \Sep_G(v_k,x)\mid \{q\in \{i,j\}\mid v_q\in S\} = Z\}$}.
It follows that $\phi$ is bijective and thus, using also Proposition~\ref{prop:monotonicity}, we infer that
$|\Sep_G(v_k,x)| = |\Sep_{G-\{v_i,v_j\}}(v_k,x)|\le s(G-\{v_i,v_j\})\le
s(G-v_i)\le |V(G)|-1$.
\end{sloppypar}

Finally, we consider sets of the form $\Sep_G(x,y)$ where $x\in X_i$ and $y\in X_j$ with $xy\not\in E(G)$. In this case, the function $\phi': \Sep_{G-v_k}(x,y)\rightarrow \Sep_G(x, y)$ given by the rule $\phi'(S)=S\cup \{v_k\}$ for all $S\in \Sep_{G-v_k}(x,y)$ is a bijection. Consequently, $|\Sep_G(x,y)|\le s(G-v_k)\le |V(G)|-1$.

Since the value of $s(G)$ is not larger than the sum of the cardinalities of the at most ${|V(G)|\choose 2}$ sets of minimal $a,b$-separators, over all non-adjacent vertex pairs $a,b$, we infer that $G$ has at most
${|V(G)|\choose 2}\cdot (|V(G)|+1)$ minimal separators.
This settles Case 2 and completes the proof. \qed
\end{proof}

\begin{repproposition}{proposition:4P1C4diamond}
\label{prop:4P1C4diamond}
The class of $\{4\1, C_4$, diamond$\}$-free graphs is tame.
\end{repproposition}

\begin{proof}
Let $G$ be a $\{4\1, C_4$, diamond$\}$-free graph.
By Proposition~\ref{prop:truetwin} and Corollary~\ref{cor:tame-disconnected}, we may assume that $G$ is connected and true-twin-free. Since $G$ is $4P_1$-free, it has no induced cycles of length more than $7$. So the only possible cycles in $G$ are $C_3,C_5,C_6,C_7$. We may assume that $G$ contains an induced $C_k$ for some $k\in \{5,6,7\}$, since otherwise $G$ is chordal and has a polynomially bounded number of minimal separators~\cite{MR0408312}.
We will show that $G$ has at most $14$ vertices, which will imply that $G$ has at most a constant number of minimal separators. We consider three exhaustive cases depending on which of the cycles $C_5$, $C_6$, or $C_7$ exists in $G$.

\medskip

\textit{Case 1: $G$ contains an induced cycle of length $6$.}

Let $C$ be an induced cycle of length $6$ in $G$ and let its vertices be denoted as $v_1,\ldots, v_6$ in cyclic order. Throughout this case, all indices of vertices in $C$ will be considered modulo $6$. We first analyze the possible neighborhoods in $C$ of vertices not in $C$. Since $G$ is $4P_1$-free, every vertex $v\in V(G)\setminus V(C)$ has a neighbor in $C$. We may thus assume without loss of generality that $v$ is adjacent to $v_1$. Then, $v$ is not adjacent to $v_3$ since otherwise $G$ would contain either an induced diamond (if $v$ is adjacent to $v_2$) or an induced $C_4$ (otherwise). By symmetry, we infer that $v$ is not adjacent to $v_5$. If $v$ is not adjacent to any of $\{v_2,v_4,v_6\}$, then $\{v, v_2,v_4,v_6\}$ would form an induced $4\1$ in $G$, which is not possible. Moreover, to avoid an induced copy of the diamond or of the $C_4$, we infer that $v$ is adjacent to exactly one vertex in
$\{v_2,v_4,v_6\}$. By symmetry, we may assume that $v$ is adjacent either to $v_2$ or to $v_4$.  We have shown that for every vertex $v\in V(G)\setminus V(C)$, we have that either $N(v)\cap V(C)=\{v_i,v_{i+1}\}$ or $N(v)\cap V(C)=\{v_i,v_{i+3}\}$, for some $i\in \{1,\dots,6\}$. Vertices $v\in V(G)\setminus V(C)$ satisfying the condition $N(v)\cap V(C)=\{v_i,v_{i+1}\}$ for some $i\in \{1,\ldots, 6\}$ will be said to be of \textit{type $1$}, while the others will be of \textit{type $2$}.

We claim that $G$ has at most one vertex of type $2$. Suppose this is not the case, and let $u,v\in V(G)\setminus V(C)$ be two distinct vertices of type $2$. By symmetry, we may assume that $N(u)\cap V(C)=\{v_1,v_{4}\}$ and that $N(v)\cap V(C)=\{v_i, v_{i+3}\}$ for some $i\in \{1,2\}$. If $i = 1$, then the vertex set $\{u,v_1,v,v_4\}$ would induce either a $C_4$ (if $u$ and $v$ are non-adjacent) or a diamond in $G$ (otherwise). Similarly, if $i=2$, then
$G$ contains either an induced $4\1$ with vertex set $\{u,v,v_3,v_6\}$ (if $u$ and $v$ are non-adjacent) or an induced $C_4$ with vertex set $\{u,v_1,v_2,v\}$ (otherwise). Since either case leads to a contradiction, we conclude that $G$ has at most one vertex of type $2$, as claimed.

Let us now analyze the possible adjacencies between pairs of vertices of type $1$. Let us denote by $A_i$ the set of type $1$ vertices $v$ with
$N(v)\cap V(C) =\{v_i, v_{i+1}\}$, for $i\in \{1,\dots,6\}$.
Let us consider two type $1$ vertices, say $u,v\in V(G)\setminus V(C)$ with $u\in A_1$ and $v\in A_i$ for some $i\in \{1,\dots,6\}$. If $i=1$, then $u$ and $v$ are adjacent in $G$, since otherwise we have an induced diamond. If $i=2$, then $N(u)\cap N(v)\cap V(C)=\{v_2\}$. Then, $u$ is non-adjacent to $v$ since otherwise $G$ would contain a diamond induced by the vertex set $\{v_1,v_2,u,v\}$. But now, $G$ contains an induced $4\1$ with vertices $u,v,v_4,v_6$. Thus, $i\neq 2$ and by symmetry, $i\neq 6$. Next, if $i=3$, we have that $u$ and $v$ are not adjacent in $G$, since otherwise the vertex set $\{u,v_2,v_3,v\}$ would induce a $C_4$. Finally, if $i=4$, then vertices $u$ and $v$ are adjacent in $G$, since otherwise $G$ would contain an induced $4\1$ with vertices $u,v,v_3,v_6$.

We split the rest of the proof of Case 1 into two subcases depending on whether a vertex of type $2$ exists or not. Suppose first that there exists a vertex $u$ of type $2$, with (without loss of generality) $N(u)\cap V(C)=\{v_1, v_{4}\}$. We claim that $A_i= \emptyset$ for all $i\in \{1,3,4,6\}$. By symmetry, it suffices to consider the case $i = 1$. Suppose for a contradiction that $A_1\neq\emptyset$ and let $v\in A_1$. Then, $G$ contains either an induced $4\1$ with vertex set $\{u,v,v_3,v_6\}$ (if $u$ and $v$ are non-adjacent) or an induced diamond with vertex set $\{u,v,v_1,v_2\}$ (otherwise). Thus,
$A_1=A_3=A_4=A_6=\emptyset$, as claimed.
Moreover, the above adjacency analysis implies that we have that $A_2$ and $A_5$ are cliques that are complete to each other.
Moreover, if $v\in A_2\cup A_5$, then $u$ is not adjacent to $v$, as otherwise $G$ would contain an induced $C_4$ with vertex set $\{u,v_1,v_2,v\}$ (if $i =2$) or $\{u,v_4,v_5,v\}$ (if $i =5$).
Clearly, for $i\in\{2,5\}$, any two vertices in $A_i$ are true twins in $G$.
Since $G$ is true-twin-free, we infer that $|A_2|\le 1$ and
$|A_5|\le 1$. It follows that $G$ has at most $9$ vertices.

Suppose now that there is no vertex of type $2$. In this case, we may assume that if some $A_i$ is non-empty, then $A_1\neq \emptyset$. Thus, the above analysis implies that $A_2=A_6=\emptyset$ and that either $A_4= \emptyset$
or $A_3 = A_5 = \emptyset$. Moreover, any two vertices in some set $A_i$ are
true twins, which, since $G$ is true-twin-free, implies that $|A_i|\le 1$ for all  $i\in \{1,\ldots, 6\}$. It follows that $G$ has at most $9$ vertices, which completes the proof for Case 1.
\medskip

\textit{Case 2: $G$ is $C_6$-free but contains an induced cycle of length $7$.}

Let $C$ be an induced $C_7$ in $G$, with a cyclic order of vertices $v_1,\ldots, v_7$. All indices $v_i$ will be considered modulo $7$.
Let $v\in V(G)\setminus V(C)$. With a similar analysis as in Case 1, we get that $N(v)=\{v_i,v_{i+1}, v_{i+4}\}$ for some $i\in \{1,\ldots,7\}$.
Let $A_i=\{v\in V(G)\setminus V(C) \mid N(v)\cap V(C)=\{v_i,v_{i+1},v_{i+4}\}\}$ for all $i\in \{1,\dots, 7\}$.
We have $|A_i|\le 1$ since if $u,v\in A_i$ with $u\neq v$, then the vertex set $\{u,v_i,v,v_{i+4}\}$ induces either a diamond or a $C_4$ in $G$ (depending on whether $u$ and $v$ are adjacent or not). It follows that $G$ has at most $14$ vertices, which completes the proof for Case 2.

\medskip

\textit{Case 3: $G$ is $\{C_6,C_7\}$-free but contains an induced cycle of length $5$.}

Let $C$ be an induced $C_5$ in $G$ with a cyclic order of vertices $v_1,\ldots, v_5$ (indices modulo $5$). We may assume that $G\neq C$.
Since $G$ is connected, there is a vertex $v\in V(G)\setminus V(C)$ with a neighbor in $C$. The fact that $G$ is $C_4$-free and diamond-free implies that $N(v)\cap V(C)$ is a clique. Thus, in particular, $|N(v)\cap V(C)|=1$ or $|N(v)\cap V(C)|=2$. For $i\in \{1,2\}$, we will say that a vertex $v\in N(V(C))$ is of \emph{type $i$} if $|N(v)\cap V(C)|=i$. For all $i\in \{1,\ldots, 5\}$, let $A_i=\{v\in V(G)\setminus V(C)\mid N(v)\cap V(C)=\{v_i\}\}$ and $B_i = \{v\in V(G)\setminus V(C)\mid N(v)\cap V(C)=\{v_i,v_{i+1}\}\}$. Since $G$ is $4\1$-free, $A_i$ is a clique for all $i\in \{1,\ldots, 5\}$. Moreover, since $G$ is diamond-free, $B_i$ is a clique for all $i\in \{1,\ldots, 5\}$. Let $Z = V(G)\setminus (V(C)\cup N(C))$. Since $G$ is $4\1$-free, $Z$ is a clique. Moreover, if $u\in Z$ and $v\in N(C)$, then $u$ is adjacent to $v$, since otherwise the set consisting of $u$, $v$, and two non-adjacent vertices in $V(C)\setminus N(v)$ would induce a $4\1$ in $G$. It follows that any two vertices in $Z$ are true twins. Since $G$ is true-twin-free, this implies that $|Z|\le 1$.

We claim that at most one of the sets $A_i$ is non-empty. Suppose that $A_1\neq \emptyset$. Let $u\in A_1$. We claim that $A_2 = A_3 = A_4 = A_5=\emptyset$. By symmetry, it suffices to show that
$A_2 = A_3 =\emptyset$. If there exists a vertex $v\in A_2$, then $G$ contains either an induced $C_4$ with vertex set $\{u,v,v_1,v_{2}\}$ (if $u$ and $v$ are adjacent) or an induced $4\1$ with vertex set $\{u,v,v_3,v_{5}\}$ (otherwise). Hence $A_2=\emptyset$. If there exists a vertex $v\in A_3$,
then $G$ contains either an induced $C_6$ with vertex set $\{u,v,v_3,v_4,v_5,v_1\}$ (if $u$ and $v$ are adjacent)
or an induced $4\1$ with vertex set $\{u,v,v_2,v_{4}\}$ (otherwise).
Hence $A_3=\emptyset$.

We show next that for all $i,j\in \{1,\ldots, 5\}$, $i\neq j$, sets $B_i$ and $B_{j}$ are anticomplete to each other. By symmetry, it suffices to consider the cases $(i,j) = (1,2)$ and $(i,j) = (1,3)$. If $u\in B_1$ and $v\in B_{2}$, then $u$ and $v$ are non-adjacent to each other, since otherwise $G$ would contain an induced diamond with vertex set $\{u,v,v_1,v_{2}\}$. Moreover, if $u\in B_1$ and $v\in B_{3}$, then $u$ and $v$ are non-adjacent to each other, since otherwise $G$ would contain an induced $C_4$ with vertex set $\{u,v_{2},v_{3},v\}$.

Suppose that $A_i = \emptyset$ for all $i\in \{1,\ldots, 5\}$. Then, the above analysis implies that for every $j\in \{1,\ldots, 5\}$, any two vertices in a set $B_j$ are true twins. Hence $|B_j|\le 1$ for all $j$, which implies that $G$ has at most 11 vertices. Finally, suppose that some $A_i$ is non-empty. By symmetry, we may assume that $A_1\neq \emptyset$. Let $u\in A_1$. We claim that $B_1 = \emptyset$. If not, say $v\in B_1$, then $G$ contains either an induced diamond with vertex set $\{u,v,v_1,v_{2}\}$ (if $u$ is adjacent to $v$) or an induced $4\1$ with vertex set $\{u,v,v_{3},v_{5}\}$ (otherwise). Thus $B_1= \emptyset$ and by symmetry $B_5 = \emptyset$. If $v\in B_2$, then $u$ is non-adjacent to $v$, since otherwise
$G$ would contain an induced $C_4$ with vertex set $\{u,v_1,v_{2},v\}$.
If $v\in B_3$, then $u$ is adjacent to $v$, since otherwise
$G$ would contain an induced $4\1$ with vertex set $\{u,v,v_2,v_{5}\}$.
By symmetry, this shows that for every $X\in \{A_1,B_2,B_3,B_4\}$, any two vertices in $X$ are true twins in $G$. Since $G$ is true-twin-free, this implies that $G$ has at most 10 vertices.
This completes the proof. \qed
\end{proof}

Propositions~\ref{prop:Ramsey} and~\ref{prop:3p1andF-free} can be used to prove the following.

\begin{repproposition}{proposition:k4free}
\label{prop:k4free}
For every $F\in \{4\1,\2+2\1,P_3+\1\}$, the class of $\{F, K_4\}$-free graphs is tame.
\end{repproposition}

\begin{proof}
\begin{enumerate}[i)]
\item By Ramsey's theorem, the class of $\{4\1,K_4\}$-free graph consists of finitely many graphs, so it is tame.

\item The class of $\{4\1,K_4\}$-free graphs is
a subclass of the class of $\{\2+4\1,K_4+\2\}$-free graphs, which is tame by Proposition~\ref{prop:Ramsey}.

\item By Proposition~\ref{prop:3p1andF-free}, it suffices to show that the class of $\{3\1,K_4\}$-free graphs is tame. This follows from Ramsey's theorem.\qed
\end{enumerate}
\end{proof}

\begin{proposition}
\label{prop:pawfree}
For every $F\in \{4\1$, $\2+2\1$, $P_3+\1$, claw$\}$ the class of $\{F, \textrm{paw}\}$-free graphs is tame.
\end{proposition}

\begin{proof}
Let $G$ be an $\{F, \textrm{paw}\}$-free graph.
By Corollary~\ref{cor:tame-disconnected}, we may assume that $G$ is connected. Theorem~\ref{thm:pawfree} implies that $G$ is either $K_3$-free, or complete multipartite. If $G$ is $K_3$-free, then $G$ is also $K_4$-free and Proposition~\ref{prop:k4free} implies that $G$ has a polynomially bounded number of minimal separators.
If $G$ is complete multipartite, then $G$ is $P_4$-free, and thus has a polynomially bounded number of minimal separators by Theorem~\ref{thm:cographs}. It follows that the class of $\{F$, paw$\}$-free graphs is tame.
\qed\end{proof}

In the following we give a lemma that can be used for the proof that the class of $2\2$-free graphs is tame. 

\begin{lemma}\label{lem:2\2-free-separators}
Let $G$ be a $2\2$-free graph and let $S$ be a minimal separator in $G$.
Then there exists a vertex $v\in V(G)$ such that $S = N(v)$.
\end{lemma}

\begin{proof}
By Lemma~\ref{lem:characterization-minimal-separators}, graph $G-S$ has two $S$-full components $C$ and $D$. If both $C$ and $D$ have at least two vertices, then each of them contains at least one edge. Since $C$ and $D$ are anticomplete to each other, these edges form a $2\2$ in $G$. We may thus assume, by symmetry, that $C=\{v\}$ for some $v\in V(G)$. Then it follows that $N(v)\subseteq S$ and since every vertex of $S$ is adjacent to $v$, we must have $S=N(v)$, as claimed.
\qed\end{proof}

\begin{repproposition}{proposition:2p2}
\label{prop:2p2}
The class of $2\2$-free graphs is tame.
\end{repproposition}

\begin{proof}
Immediate from Lemma~\ref{lem:2\2-free-separators}.
\qed\end{proof}

We now have all the ingredients ready to prove Theorem~\ref{thm:main}.

\repeattheorem{theorem:main}

\begin{proof}
Let $\F$ be a family of graphs on at most $4$ vertices such that
$\F\neq \{C_4,4\1\}$ and $\F\neq \{4\1$, $C_4$, claw$\}$.
If $\F'$ is a family of graphs satisfying one of the conditions
$i)-\ref{item-last})$, then the class of $\F'$-free graphs is tame by Theorem~\ref{thm:cographs} and Propositions~\ref{prop:k3+p1free},~\ref{prop:C4p3+p1},~\ref{prop:C4p2+2p1},~\ref{prop:pawfree},~\ref{prop:4P1C4diamond},~\ref{prop:k4free}, and~\ref{prop:2p2}. Thus, if $\F\trianglelefteq \F'$ for some family of graphs satisfying one of the conditions $\ref{item1})-\ref{item-last})$, then the class of $\F$-free graphs, being a subclass of the tame class of $\F'$-free graphs, is tame, too.

Suppose now that for all families $\F'$ in $\ref{item1})-\ref{item-last})$ we have $\F\not\trianglelefteq \F'$. We want to prove that the class of $\F$-free graphs is not tame. Since  $\F\not\trianglelefteq \{2\2\}$ and
$\F\not\trianglelefteq \{P_4\}$, it follows that if $F\subseteq_i2\2$ or
$F\subseteq_i P_4$, then $F\notin \F$. Let $A=\{K_3$, $C_4$, $K_3+\1$, paw, diamond, $K_4\}$, $B=\{3\1$, $4\1$, $P_2+2\1$, $P_3+\1$, claw$\}$.
Since $\F$ does not contain any induced subgraph of either $2\2$ or $P_4$, we infer that $\F\subseteq A\cup B$. Since Proposition~\ref{prop:forestcycle} implies that the class of $\F$-free graphs is not tame if all graphs in $\F$ contain cycles or all of them are acyclic, we may assume that $\F$ contains two graphs $F_1$ and $F_2$ such that $F_1$ contains a cycle and $F_2$ is acyclic. Clearly, $F_1\in A$ and $F_2\in B$.

We claim that $\F\cap \{K_3$, $K_3+\1$, paw$\} = \emptyset$. Indeed, if
$F\in \{K_3$, $K_3+\1$, paw$\}$, then $\{F,F_2\}\subseteq \F$, which implies that $\F\trianglelefteq \F'$ for $\F'=\{F'$, $F''\}$ where $F'\in \{4\1,$  $\2+2\1$, $P_3+\1,$ claw$\}$ and $F''\in \{$paw, $K_3+\1\}$, contrary to the assumptions on $\F$. It follows that $F_1\in \F\cap A\subseteq \{K_4$, $C_4$, diamond$\}$.

Suppose that $K_4\in \F$. If there exists a graph $F\in \F\cap \{3\1, 4\1,\2+2\1, P_3+\1\}$, then $\F\trianglelefteq \F'$ where $\F'$ satisfies condition $iv)$, a contradiction. It follows that $F_2\in \F\cap B\subseteq \{$claw$\}$, that is, $F_2$ is the claw. We also have $\F\setminus \{K_4$, claw$\}\subseteq \{C_4$, diamond$\}$. Consequently, $\{$claw, $K_4$, $C_4$, diamond$\}\trianglelefteq\F$. By Corollary~\ref{cor:wallcontainedinf}, the class of $\{$claw, $K_4$, $C_4$, diamond$\}$ is not tame and hence by Observation~\ref{obs:trivial}, neither is the class of $\F$-free graphs.

From now on, we assume that $K_4\not\in \F$. Suppose that $C_4\in \F$. If $\{3\1, \2+2\1, P_3+\1\}\cap \F\neq \emptyset$, then $\F\trianglelefteq \F'$ where $\F'$ satisfies condition $v)$, a contradiction. It follows that $F_2\in \F\cap B\subseteq \{4\1$, claw$\}$. Suppose first that $4\1\in \F\cap B$. If the diamond is not in $\F$, then $\F\neq \{4\1, C_4\}$ or $\F\neq \{4\1$, claw, $C_4\}$, which is impossible. Thus, the diamond is in $\F$, which implies that $\{4\1, C_4$, diamond$\}\subseteq \F$, hence
$\F\trianglelefteq \F'$ where $\F'$ satisfies condition $vi)$, a contradiction. We conclude that $4\1\not\in \F$, which implies that $\F\cap B = \{$claw$\}$. Consequently, $\F
\subseteq \{$claw, $C_4$, diamond$\}$, which implies that $\{$claw, $K_4$, $C_4$, diamond$\}\trianglelefteq\F$. By Corollary~\ref{cor:wallcontainedinf}, the class of $\{$claw, $K_4$, $C_4$, diamond$\}$ is not tame and hence by Observation~\ref{obs:trivial}, neither is the class of $\F$-free graphs.

From now on, we assume that $C_4\not\in \F$. It follows that
$F_1\in \F\cap A\subseteq \{$diamond$\}$, that is, $\F\cap A = \{$diamond$\}$. Clearly, $F_2\in \F\cap B \subseteq \{3\1$, $4\1$, $\2+2\1$, $P_3+\1$, claw$\}$, which implies that
every graph in $\F\cap B$ contains an induced $3\1$.
Consequently, $\{3\1$, diamond$\}\trianglelefteq \F$. From  Corollary~\ref{cor:h2containedinf} it follows that the class of $\{3\1,$ diamond$\}$-free graphs is not tame and by Observation~\ref{obs:trivial}, neither is the class of $\F$-free graphs.

This completes the proof.
\qed\end{proof}

Similarly, we give the proof of Theorem~\ref{thm:main-dual}.
\repeattheorem{theorem:main-dual}

\begin{proof}
Let $\F$ be a family of graphs on at most $4$ vertices such that
$\F\neq \{4\1,C_4\}$ and $\F\neq \{4\1$, claw, $C_4\}$.
If $\F'\trianglelefteq \F$, where $\F'$ is a family of graphs satisfying one of the conditions
$\ref{item1-dual})-\ref{item-last-dual})$, then the class of $\F'$-free graphs is not tame by Corollaries~\ref{cor:h2containedinf} and \ref{cor:wallcontainedinf} and by Proposition~\ref{prop:forestcycle}, respectively, so by Observation~\ref{obs:trivial} it follows that the class of $\F$-free graphs is not tame as well.

Suppose now that for all families $\F'$ in $\ref{item1-dual})-\ref{item-last-dual})$ we have $\F'\not\trianglelefteq \F$. We want to prove that the class of $\F$-free graphs is tame. If there exists some graph $F$ in $\F$ that satisfies $F\subseteq_i 2\2$ or $F\subseteq P_4$, then the class of $\F$-free graphs is contained either in the class of $2\2$-free graphs or in the class of $P_4$-free graphs and thus it is tame by Proposition~16 or Theorem~\ref{thm:cographs}, respectively.

Let $A=\{K_3$, $C_4$, $K_3+\1$, paw, diamond, $K_4\}$ and $B=\{3\1$, $4\1$, $P_2+2\1$, $P_3+\1$, claw$\}$.
Since $\F$ does not contain any induced subgraph of either $2\2$ or $P_4$, we infer that $\F\subseteq A\cup B$. We claim that $A\cap \F\neq \emptyset$ and $B\cap \F\neq \emptyset$. Suppose first that $\F\cap A=\emptyset$. Then $\F\subseteq B$ and we have that $\{3\1\}\trianglelefteq \F$, implying that $\F'\trianglelefteq \F$, where $\F'$ satisfies $\ref{item1-dual})$, a contradiction. Similarly, if $\F\cap B\neq \emptyset$, then $\F\subseteq A$ and we have that $\F'\trianglelefteq \F$, where $\F'$ satisfies $\ref{item-last-dual})$, a contradiction. It follows that $\F\cap A\neq \emptyset$ and $\F\cap B\neq \emptyset$.

Assume first that $\F\cap \{K_3$, $K_3+\1$, paw$\} \neq \emptyset$. Since $\F\cap B\neq \emptyset$, we have that the class of $\F$-free graphs is contained in some of the classes that are proved to be tame in Propositions~\ref{prop:k3+p1free} and \ref{prop:pawfree}, so it follows by Observation~\ref{obs:trivial} that the class of $\F$-free graphs is tame as well.

From now on, we assume that $\F\cap \{K_3$, $K_3+\1$, paw$\} \neq \emptyset$. As $\F\cap A\neq \emptyset$, it follows that $\{C_4$, diamond, $K_4\}\cap \F\neq \emptyset$.
Suppose that $K_4\in \F$. If $(B\setminus\{$claw$\})\cap \F\neq \emptyset$, then the class of $\F$-free graphs is contained in one of the tame classes of graphs considered in the Proposition~\ref{prop:k4free} and thus it is tame by Observation~\ref{obs:trivial}. It follows that $B\cap \F=\{$claw$\}$. Since $\F\subseteq \{$claw, $C_4,$ diamond, $K_4\}$, it follows that $\F'\trianglelefteq F$, where $\F'$ satisfies $\ref{item2-dual})$, a contradiction.

From now on, we assume that $K_4\not\in \F$. Suppose that $C_4\in \F$. If $\{3\1, \2+2\1, P_3+\1\}\cap \F\neq \emptyset$, then the class of $\F$-free graphs is contained either in the class of $\{\2+2\1$, $C_4\}$-free graphs or in the class of $\{P_3+\1$, $C_4\}$-free graphs and by Propositions~\ref{prop:C4p2+2p1}-\ref{prop:C4p3+p1} and Observation~\ref{obs:trivial} it is tame. It follows that $\F\cap B\subseteq \{4\1,$ claw$\}$.  Assume first that $4\1\in \F$. Since $\F\neq \{4\1, C_4\}$, $\F\neq \{4\1$, claw, $C_4\}$, it follows that diamond is in $\F$. Then we have that $\F\trianglelefteq \{4\1,$ claw, diamond$\}$ and by Proposition~\ref{prop:4P1C4diamond} and Observation~\ref{obs:trivial} it follows that the class of $\F$-free graphs is tame. If $4\1\notin \F$, then it follows that claw is in $\F$ and we have that $\F'\trianglelefteq \F$, where $\F'$ satisfies $\ref{item2-dual})$, a contradiction.

From now on, we assume that $C_4\not\in \F$. It follows that
$\F\cap A\subseteq \{$diamond$\}$, that is, $\F\cap A = \{$diamond$\}$. Clearly, $\F\cap B \subseteq \{3\1$, $4\1$, $\2+2\1$, $P_3+\1$, claw$\}$, which implies that
every graph in $\F\cap B$ contains an induced $3\1$.
Consequently, $\{3\1$, diamond$\}\trianglelefteq \F$, a contradiction.

This completes the proof. \qed
\end{proof}

\section{Conclusion}\label{sec:conclusion}

In this work we considered graphs with ``few'' minimal separators. Our main result was an almost complete dichotomy for the property of having
a polynomially bounded number of minimal separators within the family of
graph classes defined by forbidden induced subgraphs with at most four vertices. Two exceptional families for which the problem is still open are
the class of $\{4\1,C_4\}$-free graphs and the class of $\{4\1, \textrm{claw}, C_4\}$-free graphs. Note that the class of $\{4\1,C_4\}$-free graphs and their complements was already of interest to Erd\H{o}s, who offered \$20 to determine whether the vertex set of every $\{4\1,C_4\}$-free graph can be covered by $4$ cliques (which was resolved in the affirmative by Nagy and Szentmikl\'ossy, see~\cite{MR951359}). Moreover, the class of
$\{4\1,C_4\}$-free graphs is one of the only three graph classes defined by a set of four-vertex forbidden induced subgraphs for which the complexity of coloring is still open~\cite{MR3736912}. Some of the results given here (for example Proposition~\ref{prop:Ramsey}) are not restricted to
forbidden induced subgraphs of at most four vertices, and they might prove useful for developing more general dichotomy studies related to minimal separators.

%
%
\bibliographystyle{splncs04}
\bibliography{biblio}
\newpage
\end{document}